\documentclass[reqno]{amsart}

\usepackage{times}
\usepackage{amsmath}
\usepackage{amsthm}
\usepackage{amsfonts}
\newtheorem{thm}{Theorem}[section]
\newtheorem{lem}[thm]{Lemma}
\newtheorem{cor}[thm]{Corollary}
\newtheorem{prop}[thm]{Proposition}

\newtheorem{rem}[thm]{Remark}

\DeclareMathAlphabet{\mathpzc}{OT1}{pzc}{m}{it}

\numberwithin{equation}{section}

\newcommand{\Wqb}{W_{q,D}}
\newcommand{\Wqq}{\Wqb^{2-2/q}}
\newcommand{\Wqqp}{\Wqb^{2-2/q,+}}
\newcommand{\R}{\mathbb{R}}

\newcommand{\N}{\mathbb{N}}
\newcommand{\A}{\mathbb{A}}

\newcommand{\Wq}{\mathbb{W}_q}
\newcommand{\Wqd}{\dot{\mathbb{W}}_q^+}
\newcommand{\Lq}{\mathbb{L}_q}

\newcommand{\ml}{\mathcal{L}}
\newcommand{\mk}{\mathcal{K}}

\newcommand{\Om}{\Omega}

\newcommand{\ve}{\varepsilon}
\newcommand{\rd}{\mathrm{d}}

\newcommand{\bqn}{\begin{equation}}
\newcommand{\eqn}{\end{equation}}
\newcommand{\bqnn}{\begin{equation*}}
\newcommand{\eqnn}{\end{equation*}}
\newcommand{\bear}{\begin{eqnarray}} 
\newcommand{\eear}{\end{eqnarray}} 
\newcommand{\bean}{\begin{eqnarray*}} 
\newcommand{\eean}{\end{eqnarray*}} 
\newcommand{\bs}{\begin{split}}
\newcommand{\es}{\end{split}}

\newcommand{\dhr}{\mathrel{\lhook\joinrel\relbar\kern-.8ex\joinrel\lhook\joinrel\rightarrow}}

\title[Positive Solutions of a Reaction-Diffusion Equations with Nonlocal Initial Conditions]
{On Positive Solutions of Some System of Reaction-Diffusion Equations with Nonlocal~Initial Conditions}

\author[Ch. Walker]{Christoph Walker}

\address{Leibniz Universit\"at Hannover, Institut f\"ur Angewandte Mathematik, Welfengarten 1, D--30167 Hannover, Germany.}
\email{walker@ifam.uni-hannover.de}

\begin{document}

\begin{abstract}
The paper focuses on positive solutions to a coupled system of parabolic equations with nonlocal initial conditions. Such equations arise as steady-state equations in an age-structured predator-prey model with diffusion. By using global bifurcation techniques, we describe the structure of the set of positive solutions with respect to two parameters measuring the intensities of the fertility of the species. In particular, we establish co-existence steady-states, i.e. solutions which are nonnegative and nontrivial in both components.
\end{abstract}


\maketitle

%%%%%%%%%%%%%%%%%%%%%%%%%%%%%%%%%%%
\section{Introduction}
%%%%%%%%%%%%%%%%%%%%%%%%%%%%%%%%%%%
This paper is dedicated to solutions $u=u(a,x)\ge 0$ and $v=v(a,x)\ge 0$ to the system of parabolic equations
\begin{align}
\partial_a u-\Delta_D u&=-(\alpha_1u+\alpha_2 v)u\ ,\quad  a\in (0,a_m)\ ,\quad x\in\Om\ ,\label{7}\\
\partial_a v-\Delta_D v&=-(\beta_1 v-\beta_2 u)v\ ,\,\quad  a\in (0,a_m)\ ,\quad x\in\Om\ ,\label{8}
\end{align}
subject to the nonlocal initial conditions
\begin{align}
u(0,x)&=\eta\int_0^{a_m}b_1(a)\, u(a,x)\, \rd a\ , \quad   x\in\Om\ ,\label{9}\\
v(0,x)&=\xi\int_0^{a_m}b_2(a)\, v(a,x)\, \rd a\ ,  \quad x\in\Om\ .\label{10}
\end{align}
The operator $-\Delta_D$ in \eqref{7}, \eqref{8} stands for the negative Laplacian on $\Om$ with subscript $D$ indicating that Dirichlet conditions are imposed on the boundary $\partial\Om$. Note that, due to the nonlocal character of the initial conditions, equations \eqref{7}-\eqref{10} do not pose a proper evolution problem.

System \eqref{7}-\eqref{10} arises when studying stationary (i.e. time-independent) solutions to a particular predator-prey system with age structure of the form
\begin{align}
\partial_tu+\partial_a u-d_1\Delta_x u&=-(\alpha_1u+\alpha_2 v)u\ ,\quad t>0\ ,\quad a\in (0,a_m)\ ,\quad x\in\Om\ ,\label{1}\\
\partial_tv+\partial_a v-d_2\Delta_x v&=-(\beta_1 v-\beta_2 u)v\ ,\,\quad t>0\ ,\quad a\in (0,a_m)\ ,\quad x\in\Om\ ,\label{2}
\end{align}
for $u=u(t,a,x)\ge 0$ and $v=v(t,a,x)\ge 0$ subject to the constraints
\begin{align}
u(t,0,x)&=\int_0^{a_m}\eta\, b_1(a)\, u(t,a,x)\,\rd a\ , \quad t>0\ ,\quad x\in\Om\ ,\label{3}\\
v(t,0,x)&=\int_0^{a_m}\xi\, b_2(a)\, v(t,a,x)\, \rd a\ , \quad t>0\ ,\quad x\in\Om\ ,\label{4}
\end{align}
and Dirichlet boundary conditions. It models the situation where a prey and a predator with population densities $u$ and $v$, respectively, inhabit the same spatial region $\Om$ and both species are assumed to be structured by age $a\in (0,a_m)$ and spatial position $x\in\Om$. Here, $a_m>0$ denotes the maximal age of the species. The constants $d_1, d_2>0$ give the rate at which the species diffuse. For notational simplicity they are taken to be  $
d_1=d_2=1$ in \eqref{7}, \eqref{8}. The mortality rates in \eqref{7}, \eqref{8} and \eqref{1}, \eqref{2} are given by
$$
\mu_1(u,v):=\alpha_1u+\alpha_2 v\ ,\quad \mu_2(u,v):=\beta_1 v-\beta_2 u
$$
with positive constants $\alpha_1,\alpha_2,\beta_1$, and $\beta_2$. Equations \eqref{3}, \eqref{4} represent the age-boundary conditions and reflect that individuals with age zero are those created when a mother individual of any age $a\in (0,a_m)$ gives birth with rates $\eta b_1(a)$ and $\xi b_2(a)$, respectively. The functions $b_j=b_j(a)\ge 0$ describe the profiles of the fertility rates while the parameters $\eta, \xi>0$ measure their intensity without affecting the structure of the birth rates. We refer to \cite{WebbSpringer} for a recent survey on the formidable literature about age-structured population models. Of course, \eqref{1}-\eqref{4} represents just a simple age-structured predator-prey model with diffusion and other, in certain regards, biologically maybe more accurate models (e.g. with other mortality and birth rates or different maximal ages for prey and predator) exist as well. The main goal of the present paper is to provide a framework in which problems of this kind including nonlocal initial conditions can be treated.\\

Of particular interest when studying \eqref{7}-\eqref{10} are {\it coexistence solutions}, i.e. solutions $(u,v)$ with both components nonnegative and nonzero. 

Variants of the elliptic counterpart of equations \eqref{7}-\eqref{10} being revealed when age-structure is neglected
and also related elliptic systems for, e.g., competing or cooperative species, have attracted considerable interest in literature both in the past \cite{BlatBrown1,BlatBrown2,CantrellCosner,CosnerLazer,DancerTAMS84,DancerJDE85,Leung,SchiaffinoTesei,ZhouPao} and, more recently, \cite{CasalEilbeckLG,DancerLGOrtega,LopezGomezPardo,LopezGomezChapman,LopezGomezJDE09,LopezGomezMolinaJDE06}, though both lists are far from being complete. Methods used in the cited literature include \mbox{sub-/}supersolution methods and bifurcation techniques for different parameters in order to establish positive solutions for the elliptic equations.

The parabolic problem \eqref{7}-\eqref{10} has recently been investigated in \cite{WalkerArchMath} for slightly different mortality rates of Holling-Tanner type \eqref{A} and particular birth profiles $b_j$ of negative exponential type. To prove coexistence solutions, a bifurcation approach has been chosen with respect to the parameters $\eta$ and $\xi$. The assumption in \cite{WalkerArchMath} that there is no maximal age, i.e. $a_m=\infty$, allows one to recover the elliptic system by integrating the parabolic equations with respect to age. In the present paper with $a_m<\infty$, however, this approach is no longer possible and the analysis becomes more involved. But considering $a_m<\infty$ will allow us herein to take advantage of compact embeddings of the underlying function spaces when interpreting solutions of \eqref{7}-\eqref{10} as the zeros of some function. It thus provides a setting, where we can apply global bifurcation techniques with respect to the bifurcation parameters $\eta$ and $\xi$. This is in contrast to \cite{WalkerArchMath}, where merely local bifurcation results have been obtained. We shall give a partial, but nevertheless rather complete description of the bifurcation diagrams with respect to these parameters. Our results are inspired by those of \cite{BlatBrown1,BlatBrown2} for the correspondig elliptic system without age structure, and our method is based on the celebrated global alternative of Rabinowitz \cite{LopezGomezChapman,Rabinowitz} as well as on the local bifurcation results of Crandall-Rabinowitz \cite{CrandallRabinowitz,LopezGomezChapman}.\\

As pointed out above, the mortality rates considered in \cite{BlatBrown2,WalkerArchMath} (see also \cite{CasalEilbeckLG}) are of Holling-Tanner type, that is, roughly of the form
\bqn\label{A}
\mu_1(u,v):=\alpha_1u+\alpha_2\frac{v}{1+mu}\ ,\quad \mu_2(u,v):=\beta_1v-\beta_2\frac{u}{1+mu}\ .
\eqn
All of the present results can be deduced for these nonlinearities as well with only minor modifications.\\
We shall also mention that the birth profiles $b_1$ and $b_2$ depend on age only. In principle, a spatial dependence could be included as well, but would require some additional effort. In the present paper we investigate positive solutions to \eqref{7}-\eqref{10} in dependence of the fertility intensities $\eta$ and $\xi$. However, one might study bifurcation of equilibrium solutions with respect to other parameters as well, like $\alpha_1$ and $\beta_1$ for instance. For the case of a single equation we refer to the techniques developed in \cite{WalkerAMPA}, which may provide a template also for system \eqref{7}-\eqref{10}.

\section{Main Results}\label{sec0}

To set the stage, let $J:=[0,a_m]$ and let $\Om\subset\R^n$ be a bounded and smooth domain. Throughout this paper we assume $\alpha_1, \alpha_2 , \beta_1 ,\beta_2>0$ and  that, for $j=1,2$,
\bqn\label{11}
b_j\in L_\infty^+(J)\ ,\quad b_j(a)>0\ \text{for $a$ near $a_m$}
\eqn
are normalized such that
\bqn\label{12}
\int_0^{a_m}b_j(a)e^{-\lambda_1 a}\,\rd a=1\ ,
\eqn
where $\lambda_1>0$ denotes the principal eigenvalue of $-\Delta_D$ on $\Om$. For technical reasons we introduce the solution space
$$
\Wq:=L_q(J,\Wqb^2(\Om))\cap W_q^1(J,L_q(\Om))
$$
with $q$ sufficiently large, e.g. $q>n+2$, but point out that all our solutions will actually be smooth with respect to both variables $a$ and $x$. The space $\Wqb^2$ stands for the Sobolev space of order 2 involving the Dirichlet boundary conditions, and we write $\Wq^+$ for the nonnegative functions in $\Wq$.\\

Clearly, for any choice of $\eta$ and $\xi$, $u\equiv 0$ solves \eqref{7} subject to \eqref{9} and $v\equiv 0$ solves \eqref{8} subject to~\eqref{10}. Moreover, taking $v\equiv 0$ in \eqref{7} we obtain positive solutions for \eqref{7} subject to \eqref{9} when regarding $\eta$ as parameter (and, of course, similarly for \eqref{8} with $u\equiv 0$ subject to \eqref{10} when regarding $\xi$ as parameter):

\begin{thm}\label{T1}
For each $\eta>1$ there is a unique solution $u_\eta\in\Wq^+\setminus\{0\}$ to
\bqn\label{13}
\partial_a u-\Delta_D u=-\alpha_1u^2\ ,\quad u(0,\cdot)=\eta\int_0^{a_m}b_1(a)\, u(a,\cdot)\,\rd a\ .
\eqn
The mapping $(\eta\mapsto u_\eta)$ belongs to $C^\infty((1,\infty),\Wq)$ and $\|u_\eta\|_{\Wq}\rightarrow\infty$ as $\eta\rightarrow\infty$. If $\eta\le 1$, then \eqref{13} has no solution in $\Wq^+\setminus\{0\}$.
\end{thm}

To study the solutions of \eqref{7}-\eqref{10} we first keep $\eta$ fixed and regard $\xi$ as bifurcation parameter. We thus write $(\xi,u,v)$ for a solution and suppress $\eta$. Then Theorem~\ref{T1} provides, in addition to the trivial branch of zero solutions
$$
\mathfrak{B}_0:=\{(\xi,0,0)\,;\,\xi\in\R\}\subset\R\times\Wq^+\times \Wq^+\ ,
$$
a semi-trivial branch
$$
\mathfrak{B}_1:=\{(\xi,0,v_\xi)\,;\,\xi\in(1,\infty)\}\subset\R^+\times\Wq^+\times (\Wq^+\setminus\{0\})\ ,
$$
where $(\xi,v_\xi)$ is the solution to \eqref{8} with $u\equiv 0$ subject to \eqref{10}.
If $\eta>1$, there is another semi-trivial branch
$$
\mathfrak{B}_2:=\{(\xi,u_\eta,0)\,;\,\xi\in\R\}\subset\R\times (\Wq^+\setminus\{0\})\times\Wq^+\ 
$$
from which a branch of positive coexistence solutions bifurcates. More precisely, we have:

\begin{thm}\label{T2}
For $\eta\le 1$ there is no solution $(\xi,u,v)\in\R^+\times(\Wq^+\setminus\{0\})\times \Wq^+$ to \eqref{7}-\eqref{10}. For $\eta>1$ there is a unique value $\xi_0(\eta)>0$ such that $(\xi_0(\eta),u_\eta,0)\in\mathfrak{B}_2$ is a bifurcation point. A branch $\mathfrak{B}_3\subset \R^+\times (\Wq^+\setminus\{0\})\times(\Wq^+\setminus\{0\})$ of solutions to \eqref{7}-\eqref{10} emanates from $(\xi_0(\eta),u_\eta,0)$ satisfying the alternatives
\begin{itemize}
\item[(i)] $\mathfrak{B}_3$ joins $\mathfrak{B}_2$ with $\mathfrak{B}_1$, or
\item[(ii)] $\mathfrak{B}_3$ is unbounded in $\R^+\times (\Wq^+\setminus\{0\})\times(\Wq^+\setminus\{0\})$.
\end{itemize}
Bifurcation is to the right, i.e., $\xi>\xi_0(\eta)$ for any $(\xi,u,v)\in \mathfrak{B}_3$. If, in addition,
\bqn\label{B}
b_2\in L_1(J,(1-e^{-sa})^{-1}\rd a)
\eqn
for some $s>0$, then (ii) can only occur if $\mathfrak{B}_3$ is unbounded with respect to the parameter $\xi$, and there is $N\in (1,\infty]$ such that (i) must occur for $1<\eta<N$.
\end{thm}

The values of $N$ and of $\xi_0(\eta)$ as well as the value $\xi_1(\eta)$ of $\xi$ associated to the point where $\mathfrak{B}_3$ meets $\mathfrak{B}_1$ if alternative (i) occurs are related to the spectral radii of some compact operators and will be determined precisely (see \eqref{17}, \eqref{ooo}, and Lemma~\ref{L34}). It is worthwhile to point out that in either case of the alternatives we obtain coexistence solutions; that is, solutions $(\xi,u,v)$ with both components nonzero, i.e. $u,v$ belonging to $\Wq^+\setminus\{0\}$. For those values of $\eta$ for which alternative (ii) occurs there are coexistence solutions for any $\xi>\xi_0(\eta)$ while for those $\eta$-values leading to occurrence of alternative (i) there are coexistence solutions for $\xi_0(\eta)<\xi<\xi_1(\eta)$.

Actually, we conjecture that under the additional assumption \eqref{B}, we can take $N= \infty$ and thus $\mathfrak{B}_3$ must join $\mathfrak{B}_2$ with $\mathfrak{B}_1$ for each $\eta>1$. We refer to Remark~\ref{R2} for further details.\\

Next, we regard $\eta$ as bifurcation parameter and keep $\xi$ fixed. We thus write $(\eta,u,v)$ for a solution to \eqref{7}-\eqref{10} and suppress $\xi$. Suppose first that $\xi>1$. Then Theorem~\ref{T1} provides two semi-trivial branches
$$
\mathfrak{S}_1:=\{(\eta,u_\eta,0)\,;\, \eta>1\}\ ,\quad \mathfrak{S}_2:=\{(\eta,0,v_\xi)\,;\, \eta\in\R\}
$$
with $\mathfrak{S}_j\subset \R\times\Wq^+\times \Wq^+$. Similarly as in Theorem~\ref{T2}, a branch of positive coexistence solutions bifurcates from $\mathfrak{S}_2$. In this case, however, the branch must be unbounded:

\begin{thm}\label{T3}
For $\xi>1$ there is a unique value $\eta_0(\xi)>1$ such that $(\eta_0(\xi),0,v_\xi)\in\mathfrak{S}_2$ is a bifurcation point. An unbounded branch $\mathfrak{S}_3\subset \R^+\times (\Wq^+\setminus\{0\})\times(\Wq^+\setminus\{0\})$ of solutions to \eqref{7}-\eqref{10} emanates from $(\eta_0(\xi),0,v_\xi)$.
This bifurcation is to the right, that is, $\eta>\eta_0(\xi)$ for any $(\eta,u,v)\in \mathfrak{S}_3$. If, in addition, $b_2$ satisfies \eqref{B}
for some $s>0$, then $\mathfrak{S}_3$ is unbounded with respect to the parameter $\eta$.
\end{thm}

Note that $\mathfrak{S}_3$ consists exclusively of coexistence solutions. If $b_2$ satisfies \eqref{B}, then there is a coexistence solution for any $\xi>1$ and any $\eta>\eta_0(\xi)$. The exact value of $\eta_0(\xi)$ will be specified later in~\eqref{oo}.\\

The case $\xi<1$ is more difficult, and we obtain merely a partial result. In fact, for values of $\xi<1$ near~1 we can show that a local branch of positive solutions bifurcates from $\mathfrak{S}_1$. Observe that $\mathfrak{S}_1$ is the only semi-trivial branch in this case.

\begin{thm}\label{T4}
There is $\delta\in[0,1)$ with the property that for $\xi\in (\delta,1)$ there are a unique value $\eta_1(\xi)>1$ and $\varepsilon>0$ such that a local branch
$$
\mathfrak{S}_4:=\{(\eta,u,v)\,;\,\eta_1(\xi)<\eta<\eta_1(\xi)+\varepsilon \}
\subset \R^+\times (\Wq^+\setminus\{0\})\times(\Wq^+\setminus\{0\})
$$ 
of positive solutions to \eqref{7}-\eqref{10} bifurcates to the right from $(\eta_1(\xi),u_{\eta_1(\xi)},0)\in \mathfrak{S}_1$.
\end{thm}

Again, $\mathfrak{S}_4$ consists exclusively of coexistence solutions. The precise values of $\delta$ and $\eta_1(\xi)>1$ will be given in \eqref{oooa} and \eqref{ooooa}, respectively.
Referring to Remark~\ref{R9} we conjecture that one can take $\delta=0$ in the statement.\\

The outline of the remainder of this paper is as follows: In Section~\ref{sec2} we first provide some auxiliary results including a comparison type lemma that are helpful for the study of semi-trivial solutions. The second part of Section~\ref{sec2} includes the proof of Theorem~\ref{T1}. Section~\ref{sec 44} is dedicated to the proof of Theorem~\ref{T2}, where $\xi$ is regarded as bifurcation parameter. The proofs of Theorems~\ref{T3} and \ref{T4} about the bifurcation results with respect to the parameter $\eta$ are given in Section~\ref{sec4}.

\section{Semi-Trivial Solutions: Proof of Theorem \ref{T1}}\label{sec2}

\subsection{Notations}

Given Banach spaces $E$ and $F$ we denote the set of bounded linear operators from $E$ into $F$ by $\ml(E,F)$. We set $\ml(E):=\ml(E,E)$, and we write $\mk(E)$ for the subspace of compact linear operators thereof. If $T\in\ml(E)$ we let $r(T)$ denote its spectral radius. Suppose now that $E$ is ordered by a convex cone $E^+$. We write $\phi\ge0$ if $\phi\in E^+$ and $\phi>0$ if $\phi\in E^+$ but $\phi\not=0$. A positive operator $T\in\ml_+(E)$ is an element $T$ of $\ml(E)$ such that $T(E^+)\subset E^+$, and we express this by $T\ge 0$. Then $\mk_+(E):=\ml_+(E)\cap\mk(E)$. Assume then further that the interior $\mathrm{int}(E^+)$ of $E^+$ is non-empty. The following equivalence turns out to be very useful in many circumstances: A point $\phi\in E^+$ is a quasi-interior point (i.e. $\langle \phi',\phi\rangle >0$ for all $\phi'$ in the dual $E'$ of $E$ with $\phi'\ge 0$ and $\phi'\not= 0$) if and only if $\phi\in \mathrm{int}(E^+)$. We call $T\in\ml_+(E)$ strongly positive provided $T\phi\in\mathrm{int}(E^+)$ for $\phi\in E^+\setminus\{0\}$. Recall that the Krein-Rutman theorem ensures (since $\mathrm{int}(E^+)\not=\emptyset$) that the spectral radius $r(T)$ of a strongly positive compact operator $T\in\mk(E)$ is positive and a simple eigenvalue with positive eigenvector and a strictly positive eigenfunctional. Moreover, $r(T)>0$ is the only eigenvalue of $T$ with a positive eigenvector. We refer to, e.g., \cite[App.A.2]{ClementEtAl} and \cite[Sect.12]{DanersKochMedina} for these facts.

Recall that $\Om$ is a bounded and smooth domain of $\R^n$. We fix $q\in (n+2,\infty)$ and set, for $\kappa>1/q$,
$$
 \Wqb^\kappa:=\Wqb^\kappa (\Om):=\{u\in W_q^\kappa; u=0\ \text{on}\ \partial\Om\}\ ,
 $$
where $W_q^\kappa:=W_q^\kappa (\Om)$ stand for the usual Sobolev-Slobodeckii spaces and values on the boundary are interpreted in the sense of traces. Then $\Wqb^{2-2/q}\hookrightarrow C^1(\bar{\Om})$ by the Sobolev embedding theorem, hence the interior of the positive cone $$\Wqqp:=\Wqb^{2-2/q}\cap L_q^+$$ is non-empty. Here, $L_q^+:=L_q^+(\Om)$ is the positive cone of $L_q:=L_q(\Om)$ consisting of functions which are nonnegative a.e. Let $J:=[0,a_m]$. We put
$$
\Lq:=L_q(J,L_q)\ ,\quad \Wq:=L_q(J,\Wqb^2)\cap W_q^1(J,L_q)\ ,
$$ 
and recall that 
\bqn\label{emb}
\Wq\hookrightarrow C\big(J,\Wqb^{2-2/q}\big)\hookrightarrow C\big(J,C^1(\bar{\Om})\big)
\eqn 
according to \cite[III.Thm.4.10.2]{LQPP}. Since $\Wq\subset W_q^1(J,L_q)\hookrightarrow C^{1-1/q}(J,L_q)$, the interpolation inequality in \cite[I.Thm.2.11.1]{LQPP} yields in fact
\bqn\label{embb}
\Wq\hookrightarrow C^{1-1/q-\vartheta}(J,\Wqb^{2\vartheta})\ ,\quad 0\le \vartheta\le 1-1/q\ .
\eqn
By \eqref{emb}, the trace $\gamma_0u:=u(0)$ defines an operator \mbox{$\gamma_0\in\ml (\Wq,\Wqb^{2-2/q})$}. We then say that an operator $A\in\ml (\Wqb^2,L_q)$ has {\it maximal $L_q$-regularity } (on $J$) provided that $$(\partial_a +A,\gamma_0)\in\ml (\Wq,\Lq\times \Wqb^{2-2/q})$$ is a toplinear isomorphism. 
For the positive cone of $\Lq$ we write $\Lq^+:=L_q^+(J,L_q)$ (i.e. those functions $u\in\Lq$ for which $u(a)$ belongs to $L_q^+$ for a.a. $a\in J$). We put $\Wq^+:=\Wq\cap L_q^+(\R^+,L_q)$ and use the notation $\Wqd:=\Wq^+\setminus\{0\}$. Note that $u\in\Wq^+$ implies $u(a)\ge 0$ on $\Om$ for $a\in J$ due to \eqref{emb}.

Let $\varphi_1$ denote the strongly positive eigenfunction to the principal eigenvalue $\lambda_1>0$ of $-\Delta_D$ with $\|\varphi_1\|_\infty=1$.
\\

\subsection{Preliminaries}

If $\varrho>0$ and $h\in C^\varrho (J,C(\bar{\Om}))$, then clearly $-\Delta_D + h\in  C^\varrho (J,\ml(\Wqb^2,L_q))$ and for $a\in J$ fixed, $\Delta_D- h(a)$ is the generator of an analytic semigroup on $L_q$ with domain $\Wqb^2$. Hence, \cite[II.Cor.4.4.1]{LQPP} ensures the existence of a parabolic evolution operator
$$
\Pi_{[h]}(a,\sigma)\ ,\quad 0\le \sigma\le a\le a_m\ ,
$$
associated with $-\Delta_D + h$. That is, given $\phi\in L_q$, $w:=\Pi_{[h]}(\cdot,\sigma)\phi$ is the unique strong solution to
$$
\partial_aw-\Delta_D w+h(a)w=0\ ,\quad a\in (\sigma,a_m]\ ,\quad w(\sigma,\cdot)=\phi\ .
$$
As $\Delta_D- h(a)$ is resolvent positive for each $a\in J$, \cite[II. Sect. 6]{LQPP} and \cite[Cor.13.6]{DanersKochMedina} entail in fact that $\Pi_{[h]}(a,\sigma)\in\ml(\Wqq)$ is strongly positive for $0\le \sigma< a\le a_m$.

In the following we put
$$
H_{[h]}:=\int_0^{a_m}b_1(a)\,\Pi_{[h]}(a,0)\,\rd a\ ,\qquad \hat{H}_{[h]}:=\int_0^{a_m}b_2(a)\,\Pi_{[h]}(a,0)\,\rd a \ .
$$
Consequently, \eqref{embb} warrants that we may write any solution $(u,v)\in\Wq\times\Wq$ to \eqref{7}-\eqref{10} equivalently in the form
\begin{align}
u(a)&=\Pi_{[\alpha_1 u+\alpha_2 v]}(a,0)\, u(0)\ ,\quad a\in J\ ,  &u(0)=\eta\, H_{[\alpha_1 u+\alpha_2 v]}\, u(0)\ ,\label{darst1}\\
 v(a)&=\Pi_{[\beta_1 v -\beta_2 u]}(a,0)\, v(0)\ ,\quad a\in J\ ,  &v(0)= \xi\, \hat{H}_{[\beta_1 v -\beta_2 u]}\, v(0)\ .\label{darst2}
\end{align}
In particular observe that $u$, $v$ are nonzero and nonnegative provided that $u(0)$, $v(0)$ are nonzero and nonnegative. The following information about the spectral radii of the operators $H_{[h]}$ and $\hat{H}_{[h]}$ will be of great importance:

\begin{lem}\label{L1}
For $h\in C^\varrho (J,C(\bar{\Om}))$ with $\varrho>0$, the operator $H_{[h]}\in\mk(\Wqb^{2-2/q})$ is strongly positive. In particular, the spectral radius $r(H_{[h]})>0$ is a simple eigenvalue with an eigenvector $B_{[h]}$ belonging to $\mathrm{int}(\Wqqp)$ and a strictly positive eigenfunctional $B_{[h]}'\in \big(\Wqq\big)'$. It is the only eigenvalue of $H_{[h]}$ with a positive eigenfunction. Moreover, if $h,g\in C^\varrho (J,C(\bar{\Om}))$ with $g\ge h$ but $g\not\equiv h$, then $r(H_{[g]})<r(H_{[h]})$. The same statements hold for $\hat{H}$.
\end{lem}

\begin{proof} As $\Pi_{[h]}(a,\sigma)$ is strongly positive for $0\le \sigma< a\le a_m$, we obtain from standard regularizing effects of $\Pi_{[h]}$ and the compact embedding $\Wqb^{2\kappa}\dhr\Wqq$, $2\kappa>2-2/q$, that $H_{[h]}\in\mk(\Wqb^{2-2/q})$ is strongly positive (see \cite[Lem.2.1]{WalkerAMPA}). Due to the Krein-Rutman theorem (e.g. \cite[Thm.12.3]{DanersKochMedina}) it then remains to prove that $r(H_{[h]})$ is decreasing in $h$. \\
Let $h,g\in C^\varrho (J,C(\bar{\Om}))$ with $g\ge h$ but $g\not\equiv h$. Fix $\phi\in\Wqqp\setminus\{0\}$ and set
$$
z(a):=\Pi_{[h]}(a,0)\phi\ ,\quad w(a):=\Pi_{[g]}(a,0)\phi\ ,\qquad a\in J\ .
$$
Let $u:=z-w$. Then
$$
\partial_au-\Delta_Du+h(a)u=(g(a)-h(a))w(a)\ ,\quad u(0)=0\ ,
$$
so
\bqn\label{15}
u(a)=\int_0^a \Pi_{[h]}(a,\sigma)\,\big((g( \sigma)-h( \sigma))w( \sigma)\big)\, \rd \sigma\ge 0\ ,\quad a\in J\ .
\eqn
The strong positivity of $\Pi_{[g]}(\sigma,0)$ ensures $w(\sigma)\in\mathrm{int}(\Wqqp)$ for $\sigma\in (0,a_m]$. Since $g\not\equiv h$, there is some $\sigma_0\in J$ such that
$$
\Pi_{[h]}(a,\sigma)\,\big((g( \sigma)-h( \sigma))w( \sigma)\big)\in \mathrm{int}(\Wqqp)\ ,\quad a\in (\sigma,a_m]\ ,\quad \sigma\ \text{near}\ \sigma_0\ .
$$
This together with \eqref{11} and \eqref{15} readily imply
\bqnn
\big(H_{[h]}-H_{[g]}\big)\phi=\int_0^{a_m}b_1(a) u(a)\, \rd a\,\in \mathrm{int}(\Wqqp)\ ,\quad \phi\in\Wqqp\setminus\{0\}\ .
\eqnn
Letting $\langle \cdot, \cdot \rangle$ denote the duality pairing in $\Wqq$, we thus deduce
$$
r(H_{[h]})\langle B_{[h]}',B_{[g]}\rangle\, =\, \langle B_{[h]}', H_{[h]} B_{[g]}\rangle \, >\,  \langle B_{[h]}', H_{[g]} B_{[g]}\rangle\, =\, r(H_{[g]})\langle B_{[h]}',B_{[g]}\rangle \ .
$$
Therefore, $r(H_{[g]})<r(H_{[h]})$.
\end{proof}

The next lemma provides a comparison principle which turns out to be a key tool to handle the nonlocal initial conditions \eqref{9}, \eqref{10}. To shorten notation we set for the remainder of this section
$$
U:=\int_0^{a_m}b_1(a)\,u(a)\,\rd a\ ,\quad V:=\int_0^{a_m}b_1(a)\,v(a)\,\rd a\ ,
$$
for $u,v\in\Wq$ and we use this definition of capital letters also for other elements of $\Wq$.

\begin{lem}\label{L2}
Let $\eta>1$ and $f\in \Lq^+$. Suppose $u,v\in\Wqd$ satisfy either
$$\partial_au-\Delta_D u=-\alpha_1 u^2+f\ ,\quad u(0)\ge\eta U\ ,\qquad
\partial_av-\Delta_D v=-\alpha_1 v^2\ ,\quad v(0)=\eta V\ ,
$$
or
$$
\partial_au-\Delta_D u=-\alpha_1 u^2\ ,\quad u(0)=\eta U\ ,\qquad
\partial_av-\Delta_D v=-\alpha_1 v^2-f\ ,\quad v(0)\le \eta V\ .
$$
Then $u\ge v$.
\end{lem}

\begin{proof}
Note that for $z:=u-v$ we have
$$
\partial_a z-\Delta_D z+\alpha_1(u+v)z=f\ge 0\ ,\quad z(0)\ge \eta Z\ ,
$$
with $u+v\in\Wq$. Thus
\bqn\label{6a}
z(a)\ge \Pi_{[\alpha_1(u+v)]}(a,0)\, z(0)\ ,\quad a\in J\ ,
\eqn
and
$$
z(0)\,\ge\,\eta Z\ge \eta\int_0^{a_m}b_1(a)\,\Pi_{[\alpha_1(u+v)]}(a,0)\,\rd a\, z(0)\, =\,\eta H_{[\alpha_1(u+v)]}\, z(0)\ ,
$$
that is,
\bqn\label{7a}
\big(1-\eta H_{[\alpha_1 (u+ v)]}\big)\, z(0)\,\ge\, 0\ .
\eqn
Suppose that the first alternative of the statement holds. Then
$$
v(a)=\Pi_{[\alpha_1 v]}(a,0)\,v(0)\ ,\quad a\in J\ ,\qquad v(0)\,=\,\eta V=\eta H_{[\alpha_1 v]}\, v(0)\ ,
$$
hence $v(0)\in\mathrm{int}(\Wqqp)$ since $v\in\Wqd$. By Lemma \ref{L1}, this implies $\eta r(H_{[\alpha_1 v]})=1$. Also, due to Lemma \ref{L1} and $u\in\Wqd$, $$r(H_{[\alpha_1 v]})>r(H_{[\alpha_1 (u+ v)]})\ ,$$ whence $1>\eta r(H_{[\alpha_1 (u+ v)]})$ so that $\big(1-\eta H_{[\alpha_1 (u+ v)]}\big)^{-1}\ge 0$ (e.g. see \cite[Eq.(12.8)]{DanersKochMedina}). Recalling \eqref{7a}, it follows $z(0)\ge 0$ and then $z(a)=u(a)-v(a)\ge 0$ for $a\in J$ owing to \eqref{6a}.
If the second alternative of the statement holds, we conclude analogously.
\end{proof}

We now focus on problems of the form
\bqn\label{13a}
\partial_a u-\Delta_D u=-\alpha_1 u^2\ ,\quad u(0,\cdot)=\eta U\ .
\eqn
Observe that the comparison principle of Lemma~\ref{L2} warrants uniqueness of solutions:

\begin{cor}\label{C1}
For $\eta>1$ there is at most one solution $u=u_\eta\in\Wqd$ to \eqref{13a}. If $u_{\eta_1}, u_{\eta_2}\in\Wqd$ are solutions to \eqref{13a} with $\eta_1>\eta_2$, then $u_{\eta_1}\ge u_{\eta_2}$ with $u_{\eta_1}\not\equiv u_{\eta_2}$.
\end{cor}

The next proposition provides a global branch of positive solutions to \eqref{13} and is the basis for Theorem~\ref{T1}.

\begin{prop}\label{P1}
Problem \eqref{13a}
admits an unbounded connected set of solutions
$$
\mathcal{U}:=\{(\eta,u_\eta)\,;\, \eta\in I\}\subset (1,\infty)\times\Wqd
$$
emanating from $(1,0)$, where $I$ is an interval in $(1,\infty)$ with left endpoint $1$. There is no solution $(\eta,u_\eta)$ in $\R^+\times\Wqd$ to \eqref{13a} if~$\eta\le 1$.
\end{prop}

\begin{proof}
Let $\A(u):=-\Delta_D+\alpha_1 u$ and  $\A_*(u):=\A(u)-\A(0)=\alpha_1 u$. Given $\nu\in [0,1)$ and $r\in [0,1-\nu)$, it follows from \cite[Thm.1.1]{AmannGlasnik00} that
$\Wq\dhr W_q^r(J,\Wqb^{2\nu})$, where $\dhr$ stands for a compact embedding. Fix $\sigma, \nu, \gamma$, and $s$ such that $1/q<\sigma<1-\nu<1$ and $0<s<1-\gamma<n/2q$. Then, by Sobolev's embedding theorem,
\bqn\label{E}
\Wq\dhr W_q^\sigma(J, W_q^{2\nu})\hookrightarrow L_\infty(J,W_q^{2\nu})\ ,\qquad \Wq\dhr W_q^s(J, W_q^{2\gamma})\hookrightarrow L_q(J,C(\bar{\Om}))\ ,
\eqn
from which we easily deduce that
$$
\A_*\in C^1\big(W_q^\sigma(J, W_q^{2\nu}),\ml(W_q^s(J, W_q^{2\gamma}),\Lq)\big)\ .
$$
Observe that $\A(0)=-\Delta_D$ has maximal $L_q$-regularity and that assumption \eqref{12} implies
\mbox{$H_{[0]}\varphi_1=\varphi_1$}
so that $r(H_{[0]})=1$ by Lemma~\ref{L1}. We are therefore in a position to apply \cite[Prop.2.5, Thm.2.7]{WalkerJDE} and conclude the existence of an unbounded connected branch $\mathcal{U}$ of solutions in $(0,\infty)\times\Wqd$ emanating from $(1,0)$. If $(\eta,u)$ is a solution to \eqref{13a} with $u\in \Wqd$, then
$z'(a)\le-\lambda_1z(a)$ for $a\in J$, where
$$
z(a):=\int_\Om \varphi_1\,u(a)\,\rd x\ ,\quad a\in J\ ,
$$
and thus
$$
z(0)=\eta\int_0^{a_m} b_1(a)\int_\Om \varphi_1 u(a)\,\rd a\,\rd x\le \eta\int_0^{a_m}b_1(a)e^{-\lambda_1 a}\,\rd a\, z(0)\ .
$$
Since $u\in \Wqd$, this inequality is actually strict and $u(0)>0$ by \eqref{darst1} (with $v\equiv 0$). Therefore, we have $z(0)>0$ and so $\eta> 1$ by the above inequality and \eqref{12}. This proves the assertion.
\end{proof}

\begin{rem}\label{RR}
Using \eqref{darst1} we have for $(\eta,u_\eta)\in\mathcal{U}$ that 
$$
u_\eta(a)=\Pi_{[\alpha_1u_\eta]}(a,0)\,u_\eta(0)\ ,\quad a\in J\ ,\qquad u_\eta(0)=\eta U_\eta =\eta\, H_{[\alpha_1u_\eta]} u_\eta(0)\ .
$$
Since $u_\eta(0)\in\Wqqp$ and $u_\eta(0)\not=0$, this implies
\bqn\label{sp}
r\big(\eta H_{[\alpha_1u_\eta]}\big)
=1
\eqn
according to Lemma~\ref{L1}.
\end{rem}

Classical regularity theory for the heat equation ensure that $u_\eta$ is smooth both with respect to $a$ and $x$ for $(\eta,u_\eta)\in\mathcal{U}$. To conclude Theorem~\ref{T1} it remains to show that the branch $\mathcal{U}$ is unbounded with respect to the parameter~$\eta$. We will need some further auxiliary results. First, we give lower and upper bounds for solutions to \eqref{13a}.

\begin{lem}\label{L3}
If $(\eta,u_\eta)\in\mathcal{U}$, then
$$
u_\eta(a)\ge\frac{\lambda_1}{\alpha_1}\frac{\eta-1}{\eta(e^{\lambda_1 a}-1)+1-e^{-\lambda_1 (a_m-a)}}\,\varphi_1\quad \text{on}\ \,\Om ,\quad a\in J\ .
$$
Moreover, there is $\kappa>0$ such that
$$
\|u_\eta(a)\|_\infty\le \frac{1}{\alpha_1 a+(\kappa\eta^2)^{-1}}\ ,\quad a\in J\ ,
$$
for $(\eta,u_\eta)\in\mathcal{U}$.
\end{lem}

\begin{proof}
Let $(\eta,u_\eta)\in\mathcal{U}$ be fixed and put
$$
c_0:=\frac{\alpha_1}{\lambda_1}\frac{\eta-e^{-\lambda_1 a_m}}{\eta-1} >\frac{\alpha_1}{\lambda_1}\ .
$$
Then
\bqn\label{12a}
\frac{c_0\lambda_1-\alpha_1}{c_0\lambda_1e^{\lambda_1 a}-\alpha_1}\,\ge\ \frac{1}{\eta e^{\lambda_1 a}}\ ,\qquad c_0\lambda_1e^{\lambda_1 a}-\alpha_1\ge c_0\lambda_1-\alpha_1>0\ ,
\eqn
for $a\in J$. Thus, $z:=f\varphi_1\in\Wq^+$, where
$$
f(a):=\frac{\lambda_1}{c_0\lambda_1e^{\lambda_1 a}-\alpha_1}\ ,\quad a \in J\ ,
$$
solves the ode
$f'+\lambda_1f=-\alpha_1 f^2$. Since $z=f\varphi_1\le f$, we obtain
$$
\partial_az-\Delta_D z=-\alpha_1 z^2-F\ ,\quad F:= \alpha_1 u(f-z)\ge 0\ .
$$
Also observe that, by \eqref{12} and \eqref{12a},
$$
1=\eta\int_0^{a_m}b_1(a)\frac{1}{\eta e^{\lambda_1 a}}\,\rd a\le \eta \int_0^{a_m}b_1(a)\frac{c_0\lambda_1-\alpha_1}{c_0\lambda_1e^{\lambda_1 a}-\alpha_1}\,\rd a\ ,
$$
whence
$$
z(0)=\frac{\lambda_1}{c_0\lambda_1-\alpha_1}\varphi_1 \le \eta \int_0^{a_m}b_1(a)\frac{\lambda_1}{c_0\lambda_1e^{\lambda_1 a}-\alpha_1}\,\rd a\,\varphi_1\, =\,\eta Z\ .
$$
Now the comparison principle of Lemma~\ref{L2} implies $u_\eta\ge z$ and the lower bound on $u_\eta$ follows from the definition of $z$.

For the second assertion set
$$
\psi(a):=\frac{1}{\alpha_1 a+\|u_\eta(0)\|_\infty^{-1}}\ ,\quad a\in J\ ,
$$
for $(\eta,u_\eta)\in\mathcal{U}$ given. Then 
$$
\psi'=-\alpha_1\psi^2\ ,\quad \psi(0)=\|u_\eta(0)\|_\infty\ge u(0)\quad\text{on}\ \Om\ .
$$
Let $w:=\psi-u_\eta$. Clearly, $w\in C^{1,2}(J\times\bar{\Om})$ and
\begin{gather*}
\partial_aw-\Delta_Dw=-\alpha_1(\psi+u_\eta)w\quad\text{on}\ J\times\Om\ ,\\
w(0,\cdot)\ge 0\quad\text{on}\ \Om\ ,\qquad w(a,\cdot)=\psi(a)>0 \quad\text{on}\ \partial\Om\ ,\quad a\in J\ .
\end{gather*}
Hence, the parabolic maximum principle (e.g. see \cite[Thm.13.5]{DanersKochMedina}) yields $w\ge 0$ on $J\times\bar{\Om}$, that is,
\bqn\label{155}
u_\eta(a,x)\le \psi(a)\ ,\quad (a,x)\in J\times\bar{\Om}\ .
\eqn
Using this we derive from the initial condition $u_\eta(0)=\eta U_\eta$ that
$$
\|u_\eta(0)\|_\infty\,\le\,\eta\|b_1\|_\infty\int_0^{a_m}\big(\alpha_1 a+\|u_\eta(0)\|_\infty^{-1}\big)^{-1}\,\rd a\,=\,\frac{\eta\|b_1\|_\infty}{\alpha_1}\,\log \big(\alpha_1 a_m\|u_\eta(0)\|_\infty+1\big)
$$
from which we easily deduce $\|u_\eta(0)\|_\infty\le (\kappa \eta)^2$ for some $\kappa>0$. Combining this with estimate \eqref{155}, we conclude also the upper bound on $u_\eta$.
\end{proof}

\subsection{Proof of Theorem~\ref{T1}}

To finish the proof of Theorem~\ref{T1} note first that, owing to Proposition~\ref{P1}, problem \eqref{13a} does not admit a solution $u$ in $\Wqd$ if $\eta\le 1$. Also recall that, again by Proposition~\ref{P1}, there is an unbounded connected branch $\mathcal{U}$ of solutions to \eqref{13a} and that uniqueness of solutions is provided by Corollary~\ref{C1}. In particular, there are $(\eta_j,u_{\eta_j})\in\mathcal{U}$ with $\|(\eta_j,u_{\eta_j})\|_{\R\times \Wq}\rightarrow\infty$ as $j\rightarrow\infty$. Since $\mathcal{U}$ is connected, the existence of a unique solution $u_\eta\in \Wqd$ to \eqref{13a} for each value $\eta>1$ will be established provided we can show that $\eta_j \rightarrow\infty$. Suppose otherwise, i.e. let $\eta_j\le\eta_*$ for some $\eta_*>1$. Then necessarily $\|u_{\eta_j}\|_{\Wq}\rightarrow\infty$. However, Lemma~\ref{L3} implies
\bqn\label{2a}
\|u_{\eta_j}(a)\|_\infty\le\kappa\eta_*^2\ ,\quad a\in J\ ,\quad j\in\N\ ,
\eqn
for some $\kappa>0$. The positivity of $u_{\eta_j}$ and \eqref{13a} ensure $0\le u_{\eta_j}(a)\le u_{\eta_j}(0)$ on $\Om$ for $a\in J$, and thus
$$
\|u_{\eta_j}^2\|_{\Lq}^q\,=\, \int_0^{a_m}\int_\Om(u_{\eta_j}(a))^{2q}\,\rd x\,\rd a\,\le\, a_m\,\|u_{\eta_j}(0)\|_{L_{2q}}^{2q}\ ,\quad j\in\N\ .
$$
Using the property of maximal $L_q$-regularity for $-\Delta_D$ in \eqref{13a}, it follows that
\bqn\label{2aa}
\|u_{\eta_j}\|_{\Wq}\,\le\, c\,\big(\|\alpha_1u_{\eta_j}^2\|_{\Lq}+\|u_{\eta_j}(0)\|_{\Wqb^{2-2/q}}\big)\,\le\,
c\,\big(\|u_{\eta_j}(0)\|_{L_{2q}}^{2}+\|u_{\eta_j}(0)\|_{\Wqb^{2-2/q}}\big)
\eqn
for $j\in\N$ and some constant $c$ independent of $u_{\eta_j}$. Writing the solution $u_{\eta_j}$ to \eqref{13a} in the form
$$
u_{\eta_j}(a)=e^{a\Delta_D}\,u_{\eta_j}(0)-\alpha_1\int_0^a e^{(a-\sigma)\Delta_D}\, (u_{\eta_j}(\sigma))^2\,\rd \sigma\ ,
$$ 
we see that
$$
u_{\eta_j}(0)=\eta_j\int_0^{a_m}b_1(a)\,
e^{a\Delta_D}\,u_{\eta_j}(0)\,\rd a-\alpha_1\,\eta_j\int_0^{a_m} b_1(a)\int_0^a e^{(a-\sigma)\Delta_D}\, (u_{\eta_j}(\sigma))^2\,\rd \sigma\,\rd a\ .
$$
Taking into account that $\|e^{a\Delta_D}\|_{\ml(L_q,\Wqq)}\le c a^{1/q-1}$ for $a>0$, e.g. due to \cite{LQPP}, we derive
from \eqref{2a} that $(u_{\eta_j}(0))_{j\in\N}$ stays bounded in $\Wqq$. But then $(u_{\eta_j})_{j\in\N}$ stays bounded in $\Wq$ by \eqref{2aa} in contradiction to our observation above. Therefore, $\eta_j \rightarrow\infty$ and we conclude that \eqref{13a} admits for each value of $\eta>1$ a unique solution $u_\eta\in\Wqd$.\\

Next, we show that $\|u_\eta\|_{\Wq}\rightarrow\infty$ as $\eta\rightarrow\infty$. Indeed, if $\|u_\eta\|_{\Wq}\le c<\infty$ for all $\eta>1$, then $\|u_\eta(0)\|_{\infty}$ would be bounded with respect to $\eta$ by \eqref{emb}. Thus $u_\eta(0)=\eta U_\eta$ would imply that $\|U_\eta\|_\infty$ tends to zero as $\eta\rightarrow \infty$ contradicting the fact
$$
\frac{\lambda_1}{\alpha_1(1-e^{-\lambda_1 a_m})}\,\frac{\eta-1}{\eta}\,\varphi_1\,\le\,\frac{1}{\eta}\, u_\eta(0)\,=\,U_\eta\quad \text{on}\ \Om
$$
and $\|\varphi_1\|_\infty=1$ according to Lemma~\ref{L3}.\\

Finally, it remains to prove that $(\eta\mapsto u_\eta)\in C^\infty((1,\infty),\Wq)$. For, set
$$
\Gamma(\eta,u):=\big(\partial_au-\Delta_Du+\alpha_1 u^2 \,,\,  u(0)-\eta U\big)
$$
and note that $\Gamma\in C^\infty ((1,\infty)\times \Wq,\Lq\times\Wqq)$ with $\Gamma(\eta,u_\eta)=(0,0)$ for $\eta>1$. In fact, if $\eta>1$ and $\phi\in\Wq$, then
$$
\Gamma_u(\eta,u_\eta)\phi=\big(\partial_a\phi-\Delta_D\phi+2\alpha_1u_\eta \phi \,,\,  \phi(0)-\eta \Phi\big)\ .
$$
Thus, $\Gamma_u(\eta,u_\eta)\phi=(\psi,\Theta)$ with $(\psi,\Theta)\in \Lq\times\Wqq$ if and only if
$$
\phi(a)=\Pi_{[2\alpha_1u_\eta]}(a,0)\phi(0)+\int_0^a \Pi_{[2\alpha_1u_\eta]}(a,\sigma)\,\psi(\sigma)\,\rd \sigma\ ,\quad a\in J\ ,
$$
and
$$
\big(1-\eta H_{[2\alpha_1u_\eta]}\big)\, \phi(0)\,=\,\eta\int_0^{a_m} b_1(a)\int_0^a \Pi_{[2\alpha_1u_\eta]}(a,\sigma)\,\psi(\sigma)\,\rd \sigma\,+\,\Theta\ .
$$
Invoking \eqref{sp} and Lemma~\ref{L1} we see that $1>r\big(\eta H_{[2\alpha_1u_\eta]}\big)$, whence $1-\eta H_{[2\alpha_1u_\eta]}$ is invertible. This readily implies that $\Gamma_u(\eta,u_\eta)$ is bijective and so $\Gamma_u(\eta,u_\eta)\in\ml(\Wq,\Lq\times\Wqq)$ is an isomorphism
by the open mapping theorem. The implicit function theorem then yields some $\ve>0$ and a function $\zeta\in C^\infty((\eta-\ve,\eta+\ve),\Wq)$ such that $\zeta(\eta)=u_\eta$ and $\Gamma(s,\zeta(s))=0$ for $\vert s-\eta\vert<\ve$. Since the solution to $\Gamma(s,u)=0$ is unique by Corollary~\ref{C1}, we derive $\zeta(s)=u_s$
and so $(\eta\mapsto u_\eta)\in C^\infty((1,\infty),\Wq)$. This completes the proof of Theorem~\ref{T1}.\\

Actually, we can say more about the derivative of $u_\eta$ with respect to $\eta$. Set $z:=\frac{\partial}{\partial\eta}u_\eta$. Differentiation of the equation
$$
\partial_au_\eta-\Delta_Du_\eta=-\alpha_1u_\eta^2\ ,\quad u_\eta(0)=\eta U_\eta
$$
with respect to $\eta$ and interchange of the smooth derivatives yield
$$
\partial_az-\Delta_Dz=-2\alpha_1u_\eta z\ ,\quad z(0)= U_\eta+\eta Z\ ,
$$
whence
$$
z(a)=\Pi_{[2\alpha_1 u_\eta]}(a,0) z(0)\ ,\quad a\in J\ ,\quad \big(1-\eta H_{[2\alpha_1u_\eta]}\big) z(0)=U_\eta\ .
$$
Since, as above, $1-\eta H_{[2\alpha_1u_\eta]}$ is invertible, we conclude
$$
z(0)=\big(1-\eta H_{[2\alpha_1u_\eta]}\big)^{-1} U_\eta\in\mathrm{int}(\Wqqp)
$$
and thus

\begin{cor}\label{C2}
If $\eta>1$, then $\frac{\partial}{\partial\eta}u_\eta (a)\in\mathrm{int}(\Wqqp)$ for $a\in J$.
\end{cor}
\bigskip

\subsection{Further Auxiliary Results}

We end this section with two results regarding nontrivial nonnegative solutions to \eqref{7}-\eqref{10}. Given $\eta, \xi>1$, let $u_\eta\in\Wqd$ denote the unique solution to \eqref{7}, \eqref{9} with $v\equiv 0$ and, accordingly, let $v_\xi\in\Wqd$ denote the unique solution to \eqref{8}, \eqref{10} with $u\equiv 0$, both solutions being provided by Theorem~\ref{T1}.

\begin{lem}\label{L4}
Let $\xi, \eta>1$ be given and suppose that $(u,v)\in \Wq^+\times \Wq^+$ solves \eqref{7}-\eqref{10}. Then
$$
0\le u(a)\le u_\eta(a)\quad\text{on}\ \Om\ ,\quad a\in J\ ,
$$
and if $v\in\Wqd$, then
$$
 v(a)\ge v_\xi(a)\quad\text{on}\ \Om\ ,\quad a\in J\ .
$$
\end{lem}

\begin{proof}
Since $u,v\in \Wq^+$, we have
$$
\partial_a u-\Delta_Du=-\alpha_1 u^2-\alpha_2uv\le -\alpha_1 u^2\ ,\quad u(0)=\eta \int_0^{a_m} b_1(a) u(a)\,\rd a\ ,
$$
and so $u(a)\le u_\eta(a)$ for $a\in J$ by Lemma~\ref{L2}. Similarly,
$$
\partial_a v-\Delta_Dv=-\beta_1 v^2+\beta_2uv\ge -\beta_1 v^2\ ,\quad v(0)=\xi \int_0^{a_m} b_2(a) v(a)\,\rd a\ ,
$$
and so $v(a)\ge v_\xi(a)$ for $a\in J$ if $v\not\equiv 0$.
\end{proof}

Next we give constraints on the parameters $\eta$ and $\xi$ for solutions to \eqref{7}-\eqref{10}.

\begin{lem}\label{L5}
Let $\xi, \eta>0$ be given and suppose that $(u,v)\in \Wq^+\times \Wq^+$ solves \eqref{7}-\eqref{10}.
\begin{itemize}
%\item[(i)] %If $u\not\equiv 0$, then
%$$
%\eta\ge \int_0^{a_m}b_1(a)e^{-\lambda_1 a}\,\rd a =1\ .
%$$
\item[(i)] If $\eta>1$ and $v\not\equiv 0$, then
\bqnn
 \xi\ge \frac{1}{r(\hat{H}_{[-\beta_2u_\eta]})}\in (0,1)\ .
\eqnn
\item[(ii)] If $\xi>1$ and $u\not\equiv 0$, then $\eta\ge 1$, and if also $v\not\equiv 0$, then
\bqnn
 \eta\ge \frac{1}{r(H_{[\alpha_2 v_\xi]})}\in (1,\infty)\ .
\eqnn
\end{itemize}
\end{lem}

\begin{proof}
(i) It follows from Lemma~\ref{L4} that
$$
\partial_a v-\Delta_Dv=-\beta_1 v^2+\beta_2uv\le \beta_2u_\eta v\ ,\quad v(0)=\xi V\ ,
$$
and so $v(a)\le \Pi_{[-\beta_2u_\eta]}(a,0) v(0)$ for $a\in J$. Hence
$$
v(0)\le \xi \int_0^{a_m}b_2(a)\,\Pi_{[-\beta_2u_\eta]}(a,0)\,\rd a\, v(0)=\xi \,\hat{H}_{[-\beta_2u_\eta]}\,v(0)
$$
i.e. $(1-\xi \hat{H}_{[-\beta_2u_\eta]})v(0)\le 0$. Suppose $\xi^{-1}> r(\hat{H}_{[-\beta_2u_\eta]})$. Then $1$ belongs to the resolvent set of $\xi \hat{H}_{[-\beta_2u_\eta]}$, whence $(1-\xi \hat{H}_{[-\beta_2u_\eta]})^{-1}\ge 0$ by \cite[Eq.(12.8)]{DanersKochMedina} yielding $v(0)\le 0$. Since $v\in\Wq^+$ by assumption, this gives $v(0)=0$ and so $v\equiv 0$ from \eqref{darst2}. From Lemma~\ref{L1} and \eqref{12} we deduce $ r(\hat{H}_{[-\beta_2u_\eta]})> r(\hat{H}_{[0]})=1$.

(ii) The first assertion is shown as in the last step of Proposition~\ref{P1}. Since 
$$
\partial_a u-\Delta_Du=-\alpha_1 u^2-\alpha_2uv\le -\alpha_2u v_\xi\ ,\quad u(0)=\eta U\ ,
$$
by Lemma~\ref{L4} if $v\not\equiv 0$, we conclude the second assertion as in (i).
\end{proof}

\section{Bifurcation for the Parameter $\xi$: Proof of Theorem \ref{T2}}\label{sec 44}

In this section we present the proof of Theorem~\ref{T2}. Regarding $\xi$ as bifurcation parameter in \eqref{7}-\eqref{10} and keeping $\eta$ fixed, we write $(\xi,u,v)$ for a solution to \eqref{7}-\eqref{10} and thus suppress $\eta$ since no confusion seems likely. First recall that Theorem~\ref{T1} warrants for any value of $\eta$ the existence of the semi-trivial branch
$$
\mathfrak{B}_1=\{(\xi,0,v_\xi)\,;\,\xi\in(1,\infty)\}\subset\R^+\times\Wq^+\times \Wqd\ ,
$$
where $(\xi,v_\xi)$ is the unique solution to \eqref{8} with $u\equiv 0$ subject to \eqref{10}. In addition,
if $\eta>1$, then there is another semi-trivial branch
$$
\mathfrak{B}_2=\{(\xi,u_\eta,0)\,;\,\xi\in\R\}\subset\R\times \Wqd\times\Wq^+\ .
$$

Let $\eta>1$ be fixed. By using Rabinowitz' global alternative \cite{LopezGomezChapman,Rabinowitz} we now show that a branch of coexistence solutions bifurcates from $(\xi_0(\eta),u_\eta,0)\in \mathfrak{B}_2$, where the choice  \bqn\label{17}
 \xi_0(\eta):=\frac{1}{r(\hat{H}_{[-\beta_2u_\eta]})}\in (0,1)\ 
\eqn
is suggested by Lemma~\ref{L5} (i). Due to Lemma~\ref{L4}, $(\xi,u,v)=(\xi,u_\eta-w,v)\in \R^+\times\Wq^+\times \Wq^+$ solves \eqref{7}-\eqref{10} if and only if $(\xi,w,v)\in \R^+\times\Wq^+\times \Wq^+$ with $w\le u_\eta$ solves
\begin{align}
&\partial_aw-\Delta_Dw=\alpha_1w^2-2\alpha_1u_\eta w+\alpha_2 (u_\eta-w) v\ ,& w(0)=\eta W\ ,\label{18a}\\
&\partial_av-\Delta_Dv=-\beta_1v^2+\beta_2 (u_\eta- w) v\ ,& v(0)=\xi V\ ,\label{18b}
\end{align}
where we slightly abuse notation by writing
$$
W:=\int_0^{a_m} b_1(a)\, w(a)\,\rd a\ ,\qquad V:=\int_0^{a_m} b_2(a)\, w(a)\,\rd a\ 
$$
when $w,v\in\Wq$. We shall use this notation also for other capital letters since it will always be clear from the context, which of the profiles $b_1$ or $b_2$ is meant. Since the interval $J$ is compact and $u_\eta\in\Wq$, it follows from \eqref{emb} and \cite[I.Cor.1.3.2,III.Thm.4.8.7,III.Thm.4.10.10]{LQPP} that
\begin{align*}
&Z_1:=\big(\partial_a-\Delta_D+2\alpha_1u_\eta,\gamma_0)^{-1}\in \ml(\Lq\times\Wqq, \Wq)\ ,\\
&Z_2:=\big(\partial_a-\Delta_D-\beta_2u_\eta,\gamma_0)^{-1}\in \ml(\Lq\times\Wqq, \Wq)\ ,
\end{align*}
due to maximal regularity. Equations \eqref{18a}, \eqref{18b} may then be restated equivalently as
\bqn\label{19}
(w,v)-K(\xi)(w,v)+R(w,v)=0\ 
\eqn
by setting
$$
K(\xi)(w,v):=\left(\begin{matrix} Z_1(\alpha_2u_\eta v,\eta W)\\ Z_2(0,\xi V)\end{matrix}\right)\ ,\qquad R(w,v):=-\left(\begin{matrix} Z_1(\alpha_1 w^2-\alpha_2wv,0)\\ Z_2(-\beta_1v^2-\beta_2wv,0)\end{matrix}\right)\ 
$$
for $(w,v)\in \Wq\times\Wq$. Obviously, $K(\xi)\in\ml(\Wq\times\Wq)$. 

\begin{lem}\label{L6}
Let $\xi\in\R$. If $\mu\ge 1$ is an eigenvalue of $K(\xi)$ with eigenvector $(w,v)\in\Wq\times\Wq$, then $\xi\not=0$, and $\mu/\xi$ is an eigenvalue of~$\hat{H}_{[-\beta_2 u_\eta]}$ with eigenvector $v(0)\in\Wqq$.
\end{lem}

\begin{proof}
Let $\mu\ge 1$ and $(w,v)\in \Wq\times\Wq\setminus\{(0,0)\}$ with $K(\xi)(w,v)=\mu (w,v)$. Suppose $v\equiv 0$. Then
$$
\partial_a w-\Delta_Dw+2\alpha_1 u_\eta w=0\ ,\quad w(0)=\frac{\eta}{\mu} W\ ,
$$
from which
$$
w(a)=\Pi_{[2\alpha_1 u_\eta]}(a,0) w(0)\ ,\quad a\in J\ ,\qquad
w(0)= \frac{\eta}{\mu} H_{[2\alpha_1 u_\eta]} w(0)\ .
$$
In particular, $w(0)\not=0$ since otherwise $(w,v)\equiv(0,0)$, and hence $\mu\le \eta r(H_{[2\alpha_1 u_\eta]})$ contradicting the fact that $1=r(\eta H_{[\alpha_1 u_\eta]})>\mu^{-1}r(\eta H_{[2\alpha_1 u_\eta]})$ by \eqref{sp} and Lemma~\ref{L1} because $\mu\ge 1$. Therefore, $v\not\equiv 0$. But from
$$
\partial_a v-\Delta_Dv-\beta_2u_\eta v=0\ ,\quad v(0)=\frac{\xi}{\mu}V\ ,
$$
it follows
$$
v(a)=\Pi_{[-\beta_2 u_\eta]}(a,0) v(0)\ ,\quad a\in J\ ,\qquad
v(0)= \frac{\xi}{\mu} \hat{H}_{[-\beta_2 u_\eta]} v(0)\ ,
$$
and so $v(0)\not= 0$ and $\xi\not=0$ since otherwise $v\equiv 0$. Consequently, $\mu/\xi$ is an eigenvalue of~$\hat{H}_{[-\beta_2 u_\eta]}$ with eigenvector $v(0)$.
\end{proof}

\begin{lem}\label{L7}
(i) $K(\xi)\in\mk(\Wq\times\Wq)$, $\xi\in\R$, is a continuous family of compact operators.

\noindent (ii) $R\in C(\Wq\times\Wq,\Wq\times\Wq)$ is compact with $R(w,v)=o\big(\|(w,v)\|_{\Wq\times\Wq}\big)$ as $(w,v)\rightarrow (0,0)$.

\noindent (iii) The set $\Sigma:=\{\xi\in\R\,;\, \mathrm{dim}\big(\mathrm{ker}(1-K(\xi))\big)\ge 1\}$ is discrete.
\end{lem}

\begin{proof}
It follows from \eqref{E} that the mapping
\bqn\label{comp}
\Wq\times\Wq\rightarrow \Lq\ ,\quad (w,v)\mapsto wv\quad\text{is compact}\ ,
\eqn
and, since $\Wqb^2\dhr\Wqq$, we easily deduce that $K(\xi)\in\ml(\Wq\times\Wq)$ and $R\in C(\Wq\times\Wq,\Wq\times\Wq)$ are compact. Finally, if $\xi\in\Sigma$, then $\mu=1$ is an eigenvalue of $K(\xi)$ and so $1/\xi$ is an eigenvalue of~$\hat{H}_{[-\beta_2 u_\eta]}$ due to Lemma~\ref{L6}. But the spectrum of the compact operator $\hat{H}_{[-\beta_2 u_\eta]}$ is discrete.
\end{proof}

In order to apply the global alternative of Rabinowitz, the next lemma will be fundamental. For a summary about the fixed point index we refer, e.g., to \cite[Sect.5.6]{LopezGomezChapman}.

\begin{lem}\label{L8}
Let $\xi_0(\eta)$ be defined in \eqref{17}. Then the fixed point index $\mathrm{Ind}(0,K(\xi))$ of zero with respect to $K(\xi)$ changes sign as $\xi$ crosses $\xi_0(\eta)$.
\end{lem}

\begin{proof}
Recall that $\mathrm{Ind}(0,K(\xi))=(-1)^{\zeta(\xi)}$, where $\zeta(\xi)$ is the sum of the algebraic multiplicities of all real eigenvalues of $K(\xi)$ greater than one. First, let $\xi<\xi_0(\eta)$ and suppose there is an eigenvalue \mbox{$\mu\ge 1$} of $K(\xi)$. Then, since $\mu/\xi$ is an eigenvalue of~$\hat{H}_{[-\beta_2 u_\eta]}$ according to Lemma~\ref{L6}, we get from \eqref{17} the contradiction $\mu/\xi\le \xi_0(\eta)^{-1}$. Thus 
$$
\mathrm{Ind}(0,K(\xi))=1\ ,\quad \xi<\xi_0(\eta)\ .
$$ 
Next, observe that, since $\hat{H}_{[-\beta_2 u_\eta]}$ is compact and strongly positive, there is some $\ve>0$ such that the interval $(\xi_0(\eta)^{-1}-\ve,\infty)$ contains only one eigenvalue of $\hat{H}_{[-\beta_2 u_\eta]}$, namely the simple eigenvalue $\xi_0(\eta)^{-1}$. Fix $\xi$ such that $\xi_0(\eta)^{-1}-\ve<\xi^{-1}\le \xi_0(\eta)^{-1}$. Then there is a unique $\mu_*\ge 1$ with $\xi/\mu_*=\xi_0(\eta)$. Clearly, if $\mu\ge 1$ is an eigenvalue of $K(\xi)$, then necessarily $\mu=\mu_*$. We claim that $\mu_*$ is a simple eigenvalue of $K(\xi)$. Indeed, since $\mu_*/\xi=r(\hat{H}_{[-\beta_2 u_\eta]})$, we may choose 
$\psi_0\in\mathrm{int}(\Wqqp)$ with $\mu_*\psi_0=\xi\hat{H}_{[-\beta_2 u_\eta]}\,\psi_0$. Setting 
$$
\psi_*:=Z_2(0,\psi_0)=\Pi_{[-\beta_2 u_\eta]}(\cdot,0)\, \psi_0\in \Wqd\ ,
$$
we obtain
\bqn\label{24}
\mu_*\psi_*=Z_2(0,\xi \Psi_*)
\eqn
as in the proof of Lemma~\ref{L6}. We then seek $\phi_*\in\Wq$ with $\mu_*\phi_*=Z_1(\alpha_2u_\eta\psi_*,\eta \Phi_*)$, i.e. a solution to
$$
\partial_a\phi_*-\Delta_D\phi_*+2\alpha_1u_\eta\phi_*=\frac{\alpha_2}{\mu_*}u_\eta\psi_*\ ,\quad \phi_*(0)=\frac{\eta}{\mu_*}\Phi_*\ 
$$
or, equivalently,
$$
\phi_*(a)=\Pi_{[2\alpha_1 u_\eta]}(a,0)\phi_*(0)+ (N\psi_*)(a)\ ,\qquad (N\psi_*)(a):=
\frac{\alpha_2}{\mu_*}\int_0^a \Pi_{[2\alpha_1 u_\eta]}(a,\sigma)\,\big(u_\eta(\sigma) \psi_*(\sigma)\big)\,\rd \sigma
$$
for $a\in J$ with
$$
\big(1-\frac{\eta}{\mu_*}H_{[2\alpha_1 u_\eta]}\big)\, \phi_*(0)\,=\, \frac{\eta}{\mu_*}\int_0^{a_m} b_1(a)(N\psi_*)(a)\,\rd a\ .
$$
Since $\mu_*\ge 1$ it follows from \eqref{sp} and Lemma~\ref{L1} that $1-\frac{\eta}{\mu_*}H_{[2\alpha_1 u_\eta]}$ is invertible and thus the equation for $\phi_*(0)$ is uniquely solvable. Thus, define $\phi_0\in\Wqqp$ and $\phi_*\in\Wq^+$ by
\begin{align*}
\phi_0&:=\frac{\eta}{\mu_*}
\big(1-\frac{\eta}{\mu_*}H_{[2\alpha_1 u_\eta]}\big)^{-1} \int_0^{a_m} b_1(a)(N\psi_*)(a)\,\rd a\ ,\\ \phi_*&:=\Pi_{[2\alpha_1 u_\eta]}(\cdot,0)\phi_0+ N\psi_* =Z_1(N\psi_*,\phi_0)\ . 
\end{align*}
Then $K(\xi)(\phi_*,\psi_*)=\mu_*(\phi_*,\psi_*)$ and it remains to prove that $\mu_*$ is simple. Clearly, the preceding discussion shows
$$
\mathrm{ker}\big(K(\xi)-\mu_*\big)=\mathrm{span}\{(\phi_*,\psi_*)\}\ .
$$
Suppose that $(\phi_*,\psi_*)\in\mathrm{rg}(K(\xi)-\mu_*)$. Then $Z_2(0,\xi V)-\mu_*v=\psi_*$ for some $v\in\Wq$, that is,
$$
\partial_av-\Delta_D v-\beta_2 u_\eta v=-\frac{1}{\mu_*}\big(\partial_a\psi_*-\Delta_D\psi_*-\beta_2 u_\eta \psi_*\big)=0\ ,\quad v(0)=\frac{\xi}{\mu_*}V-\frac{1}{\mu_*}\psi_0\ .
$$
This readily implies
$$
\big(1-\frac{\xi}{\mu_*}\hat{H}_{[-\beta_2 u_\eta]}\big)\, v(0)=-\frac{1}{\mu_*}\psi_0
$$
so that we obtain the contradiction
$$
\psi_0\in \mathrm{ker}\big(1-\frac{\xi}{\mu_*}\hat{H}_{[-\beta_2 u_\eta]}\big)\cap \mathrm{rg}\big(1-\frac{\xi}{\mu_*}\hat{H}_{[-\beta_2 u_\eta]}\big)=\{0\}
$$
since $\mu_*/\xi=r(\hat{H}_{[-\beta_2 u_\eta]})$ is a simple eigenvalue of $\hat{H}_{[-\beta_2 u_\eta]}$. Thus $(\phi_*,\psi_*)\notin\mathrm{rg}(K(\xi)-\mu_*)$ and $\mu_*$ is indeed a simple eigenvalue of $K(\xi)$. This ensures
$$
\mathrm{Ind}(0,K(\xi))=-1\ ,\quad 0\le \xi-\xi_0(\eta)<<1\ ,
$$
and the assertion follows.
\end{proof}

Taking $\xi=\xi_0(\eta)$ and $\mu_*=1$, the proof of Lemma~\ref{L8} reveals

\begin{cor}\label{C3}
$\mu_*=1$ is a simple eigenvalue of $K(\xi_0(\eta))$. Thus
$$
\Wq\times\Wq=\mathrm{ker}\big(1-K(\xi_0(\eta))\big)\oplus \mathrm{rg}\big(1-K(\xi_0(\eta))\big)\ ,\qquad \mathrm{ker}\big(1-K(\xi_0(\eta))\big)=\mathrm{span}\{(\phi_*,\psi_*)\} 
$$
with $\psi_*=Z_2(0,\psi_0)\in\Wqd$, $\psi_0=\xi_0(\eta)\Psi_* \in\mathrm{int}(\Wqqp)$, and $\phi_*\in\Wqd$.
\end{cor}

Owing to Lemma~\ref{L7} and Lemma~\ref{L8} we are now in a position to apply Rabinowitz' global alternative \cite[Cor.6.3.2]{LopezGomezChapman} to \eqref{19} and obtain a continuum $\mathfrak{C}$ (i.e. a closed and connected set) of solutions $(\xi,w,v)$ to \eqref{18a}, \eqref{18b} in $\R\times\Wq\times\Wq$ emanating from $(\xi_0(\eta),0,0)$
and satisfying the alternatives
\begin{itemize}
\item[(i)] $\mathfrak{C}$ is unbounded in $\R\times\Wq\times\Wq$, or
\item[(ii)] there is $\xi\in\Sigma\setminus\{\xi_0(\eta)\}$ with $(\xi,0,0)\in \mathfrak{C}$.
\end{itemize}
In addition, from Corollary~\ref{C3} and \cite[Lem.6.4.1]{LopezGomezChapman} it follows that for $(\xi,w,v)\in \mathfrak{C}$
near the bifurcation point $(\xi_0(\eta),0,0)$ we have
\bqn\label{26}
(w,v)=\ve\big((\phi_*,\psi_*)+ (y_1(\ve),y_2(\ve))\big)\ ,\quad -\ve_0<\ve<\ve_0\ ,
\eqn
for some $\ve_0>0$ and $y_j(\ve)=o(1)$ in $\Wq$ as $\ve\rightarrow 0$. Moreover, according to \cite[Thm.6.4.3]{LopezGomezChapman} and Corollary~\ref{C3}, the continuum $\mathfrak{C}$ consists of two subcontinua $\mathfrak{C}^{\pm}$ both emanating from $(\xi_0(\eta),0,0)$ such that $\mathfrak{C}^{+}$ contains those $(\xi,w,v)\in \mathfrak{C}$ with $\ve\in(0,\ve_0)$ in \eqref{26} and satisfies the same alternatives as $\mathfrak{C}$ or contains a point $(\hat{\xi},\hat{w},\hat{v})$ with $(\hat{w},\hat{v})\in\mathrm{rg}(1-K(\xi_0(\eta)))\setminus\{(0,0)\}$. We then set
$$
\mathfrak{B}_3':=\{(\xi,u_\eta-w,v)\, ;\, (\xi,w,v)\in \mathfrak{C}^+\}\setminus\{(\xi_0(\eta),u_\eta,0)\}\ .
$$
Observe that $(\xi,u,v)\in \mathfrak{B}_3'$ is a solution to \eqref{7}-\eqref{10} with $\xi>\xi_0(\eta)$ by \eqref{17} and Lemma~\ref{L5}, and close to $(\xi_0(\eta),u_\eta,0)$ we can write $(u,v)$ in the form 
$$
(u,v)=\big(u_\eta-\ve\phi_*-\ve y_1(\ve),\ve\psi_*+ \ve y_2(\ve)\big)\ ,\quad 0<\ve<\ve_0\ .
$$
In particular, since $u_\eta(0)\in\mathrm{int}(\Wqqp)$ and $y_1(\ve)=o(1)$ in $\Wq$ as $\ve\rightarrow 0$, we derive from \eqref{emb} that 
$$
u(0)=u_\eta(0)-\ve\phi_*(0)-\ve\gamma_0 y_1(\ve)\in\mathrm{int}(\Wqqp)\ ,\quad 0<\ve<\ve_0\ ,
$$
for $\ve_0>0$ sufficiently small. Hence, using \eqref{7} and \eqref{darst1},
the strong positivity of the evolution operator implies $u\in\Wqd$. Similarly, since $\psi_0\in\mathrm{int}(\Wqqp)$ and $\gamma_0 y_2(\ve)\rightarrow 0$ in $\Wqq$ as $\ve\rightarrow 0$ by \eqref{emb}, we also get 
$$
v(0)=\ve\psi_0+\ve\gamma_0 y_2(\ve) \in\mathrm{int}(\Wqqp)\ ,\quad 0<\ve<\ve_0\ ,
$$ 
with $\ve_0>0$ sufficiently small and thus $v\in\Wqd$ by \eqref{8} and~\eqref{darst2}. 

Therefore, points on the branch $\mathfrak{B}_3'$ close to $(\xi_0(\eta),u_\eta,0)\in\mathfrak{B}_2$ belong to $\R^+\times\Wqd\times \Wqd$. Furthermore, defining
$$
\mathfrak{B}_3:=\mathfrak{B}_3'\cap \big(\R^+\times\Wqd\times \Wqd\big)
$$
we have

\begin{lem}\label{L9}
The branch $\mathfrak{B}_3$ either joins $\mathfrak{B}_2$ with $\mathfrak{B}_1$, or is unbounded in $\R^+\times \Wqd\times\Wqd$.
\end{lem}

\begin{proof}
Suppose that $\mathfrak{B}_3'$ is contained in $\R^+\times \Wqd\times\Wqd$, that is, $\mathfrak{B}_3'=\mathfrak{B}_3$. Then, according to the alternatives satisfied by $\mathfrak{C}^+$, either
\begin{itemize}
\item[(i)] $\mathfrak{B}_3'$ is unbounded in $\R\times\Wq\times\Wq$, or
\item[(ii)] $\mathfrak{B}_3'$ contains a point $(\xi,u_\eta,0)$ with $\xi\in\Sigma\setminus\{\xi_0(\eta)\}$, or
\item[(iii)] $\mathfrak{B}_3'$ contains a point $(\xi,u_\eta-w,v)$ with $(w,v)\in\mathrm{rg}(1-K(\xi_0(\eta)))\setminus\{(0,0)\}$.
\end{itemize}
Clearly, since $\mathfrak{B}_3'\subset\R^+\times \Wqd\times\Wqd$ by assumption, alternative (ii) is impossible. We now show that alternative (iii) can also be ruled out. Suppose otherwise and let $(\xi,u_\eta-w,v)\in \mathfrak{B}_3'$ and $(f,g)\in\Wq\times\Wq$ with 
$$
(0,0)\not=(w,v)=(1-K(\xi_0(\eta))) (f,g)\ .
$$
As $v\in\Wqd$, we obtain from \eqref{darst2} and \eqref{8} that $v(0)=\xi V\in\mathrm{int}(\Wqqp)$. Due to Corollary~\ref{C3}, $\psi_*(0)=\psi_0\in\mathrm{int}(\Wqqp)$ and so we may choose $\tau>0$ such that $g(0)-v(0)+\tau\psi_0\in\mathrm{int}(\Wqqp)$. Note that
$$
v=g-Z_2(0,\xi_0(\eta)G)\ ,\quad \psi_*=Z_2(0,\xi_0(\eta) \Psi_*)\ ,\quad p:=g-v+\tau\psi_*=Z_2\big(0,\xi_0(\eta)(G+\tau\Psi_*)\big)\ .
$$
The last equality reads
$$
\partial_a p-\Delta_D p-\beta_2u_\eta p=0\ ,\quad p(0)=\xi_0(\eta)(G+\tau\Psi_*)=\xi_0(\eta) P +\xi_0(\eta) V\ ,
$$
from which we deduce that
\bqn\label{uu}
\big(1-\xi_0(\eta)\hat{H}_ {[-\beta_2u_\eta]}\big) p(0)=\xi_0(\eta) V \in\mathrm{int}(\Wqqp)
\eqn
with $p(0)\in\mathrm{int}(\Wqqp)$ by the choice of $\tau$. However, \eqref{uu} has no positive solution owing to \cite[Cor.12.4]{DanersKochMedina} and the definition of $\xi_0(\eta)$ in \eqref{17}. This contradiction ensures that alternative (iii) is also impossible. Consequently, if $\mathfrak{B}_3'$ is completely contained in $\R^+\times \Wqd\times\Wqd$, then $\mathfrak{B}_3'=\mathfrak{B}_3$ is necessarily unbounded. It remains to verify that if $\mathfrak{B}_3'$ is not contained in $\R^+\times \Wqd\times\Wqd$, then $\mathfrak{B}_3$ joins $\mathfrak{B}_2$ with $\mathfrak{B}_1$. 

Supposing that $\mathfrak{B}_3'$ is not completely contained in $\R^+\times \Wqd\times\Wqd$, there are $$(\xi_j,u_j,v_j)\in \R^+\times \Wqd\times\Wqd\quad \text{and}\quad (\xi,u,v)\in \mathfrak{B}_3'\ ,\quad (u,v)\notin \Wqd\times\Wqd$$ with $$(\xi_j,u_j,v_j)\rightarrow (\xi,u,v)\quad\text{in}\quad \R\times \Wq\times\Wq\ .$$ As \eqref{emb} ensures $u(0)\ge 0$ and $v(0)\ge 0$, whence $u,v\in\Wq^+$ by \eqref{darst1}, \eqref{darst2}, the only possibility that $(u,v)$ does not belong to $\Wqd\times\Wqd$ is that $u\equiv 0$ or $v\equiv 0$. 

Assume that both $u\equiv 0$ and $v\equiv 0$. Then $(\xi,u,v)=(\xi,0,0)\in \mathfrak{B}_0$. But the only nontrivial, nonnegative solutions to \eqref{7}-\eqref{10} close to $\mathfrak{B}_0$ lie on the branch $\mathfrak{B}_1=\{(\xi,0,v_\xi)\,;\,\xi\in(1,\infty)\}$, that is,
$(\xi_j,u_j,v_j)$ belong to $\mathfrak{B}_1$ which is impossible since $u_j\in \Wqd$.

Next, assume that $u\not\equiv 0$ but $v\equiv 0$. Then the uniqueness statement of Theorem~\ref{T1} yields $u=u_\eta$. So $(\xi,0,0)$ is a bifurcation point for \eqref{18a}, \eqref{18b}, or equivalently, for \eqref{19}. Thus \cite[Lem.6.1.2]{LopezGomezChapman} implies $\xi\in \Sigma$, whence $\mu=1$ is an eigenvalue of $K(\xi)$. Setting $w_j:=u_\eta-u_j$, it follows from the properties of $K(\xi)$ and $R$ stated Lemma~\ref{L7} exactly as in the proof of \cite[Lem.6.5.3]{LopezGomezChapman}
(see also \cite[Thm.3.1]{BlatBrown2}) that
$$\frac{(w_j,v_j)}{\|(w_j,v_j)\|_{\Wq\times\Wq}}$$ converges to an eigenvector $(\bar{w},\bar{v})\in \Wq^+\times\Wq^+$ of $K(\xi)$ corresponding to the eigenvalue 1. Lemma~\ref{L6} shows that $\bar{v}(0)$ is a positive eigenvector to $\hat{H}_{[-\beta_2u_\eta]}$ associated to the eigenvalue $1/\xi$ and thus $\xi=\xi_0(\eta)$ since $1/\xi_0(\eta)$ is the only eigenvalue with positive eigenvector. But then $(\xi,u,v)=(\xi_0(\eta),u_\eta,0)$ and this is not possible.

Thus, the only possibility is that $u\equiv 0$ but $v\not\equiv 0$ so that, due to the uniqueness statement of Theorem~\ref{T1}, $(\xi,u,v)=(\xi,0,v_\xi)\in\mathfrak{B}_1$. Consequently, $\mathfrak{B}_3'$ joins $\mathfrak{B}_2$ with $\mathfrak{B}_1$ and, as $\mathfrak{B}_3'$ leaves $\R^+\times \Wqd\times\Wqd$ only when meeting $\mathfrak{B}_1$, the same must be true for $\mathfrak{B}_3$.
\end{proof}

We also need to show that if $b_2$ additionally satisfies \eqref{B}, then $\mathfrak{B}_3$ can be unbounded only if it is unbounded with respect to the parameter $\xi$. This is the content of the next lemma.

\begin{lem}\label{L10}
Let $b_2$ satisfy \eqref{B}. For $M>1$ there is $c(M)>0$ such that $\|u\|_{\Wq}+\|v\|_{\Wq}\le c(M)$ whenever $(\xi,u,v)\in\R^+\times\Wq^+\times\Wq^+$ is a solution to \eqref{7}-\eqref{10} with $\xi\le M$.
\end{lem}

\begin{proof}
Let $(\xi,u,v)\in\R^+\times\Wq^+\times\Wq^+$ be any solution to \eqref{7}-\eqref{10} with $\xi\le M$. Since $$u(a)\le u_\eta(a)\le \kappa\eta^2\ ,\quad a\in J\ ,$$ by Lemma~\ref{L3} and Lemma~\ref{L4}, we have
$$
\partial_av-\Delta_Dv=-\beta_1v^2+\beta_2 uv\le -\beta_1v^2+\beta_2 \kappa\eta^2 v\ ,\qquad v(0)=\xi V\ .
$$
Put $m:=\beta_2 \kappa\eta^2$ and
$$
f(a)\,:=m\,\|v(0)\|_\infty\,\left(\beta_1\,\|v(0)\|_\infty\big(1-e^{-ma}\big)+me^{-ma}\right)^{-1}\ ,\quad a\in J\ ,
$$
so that 
\bqn\label{z}
f'=-\beta_1f^2+m f\ ,\quad f(0)=\|v(0)\|_\infty\ ,\qquad f(a)\le\|v(0)\|_\infty \, e^{ma}\ ,\quad a\in J\  .
\eqn
Let $z:=f-v$ and observe that
$$
\partial_a z-\Delta_D z\ge -\beta_1(f+v) z+m z\quad\text{on}\ J\times\Om\ ,\qquad
z\ge 0\ \, \text{on}\ \, J\times\partial\Om\ ,\qquad z(0,\cdot)\ge 0\ \, \text{on}\ \, \Om\ ,
$$
from which we get $z\ge0$, i.e. $v(a)\le f(a)$ on $\bar{\Om}$ for $a\in J$ owing to the parabolic maximum principle \cite[Thm.13.5]{DanersKochMedina}. Since we may assume that $m\ge s$ with $s$ from \eqref{B}, we have
$$
f(a)\le\frac{m}{\beta_1(1-e^{-ma})}\le \frac{m}{\beta_1(1-e^{-sa})}\ ,\quad a>0\ ,
$$
and so it follows from $\xi\le M$ that
$$
v(0)\le M\int_0^{a_m}b_2(a) f(a)\,\rd a\le \frac{Mm}{\beta_1}\int_0^{a_m}b_2(a) (1-e^{-sa})^{-1}\,\rd a\, <\, \infty\ ,
$$
whence 
\bqn\label{zz}
\|u(a)\|_\infty+\|v(a)\|_\infty\le c(M)\ ,\quad a\in J\ ,
\eqn
for some $c(M)>0$ by \eqref{z}. Next, using the maximal regularity property of $-\Delta_D$, we derive from \eqref{7} and \eqref{zz} that there is $c_0(M)>0$ such that
$$
\|u\|_{\Wq}\,\le\, c\,\big(\|u(0)\|_{\Wqq}+\|\alpha_1 u^2+\alpha_2 uv\|_{\Lq}\big)\,\le\, c_0(M)\, \big(\|u(0)\|_{\Wqq}+1\big)\ .
$$
Writing \eqref{7} in the form
$$
u(a)=e^{a\Delta_D}u(0)+\int_0^a e^{(a-\sigma)\Delta_D}\big(-\alpha_1 u(\sigma)^2-\alpha_2 u(\sigma)v(\sigma)\big)\,\rd \sigma\ , \quad a\in J\ ,
$$
and using $\|e^{a\Delta_D}\|_{\ml(L_q,\Wqq)}\le c a^{1/q-1}$ for $a>0$, we obtain from \eqref{9} and \eqref{zz}
\bqnn
\begin{split}
\|u(0)\|_{\Wqq}\,&\le\,\eta\,\|b_1\|_\infty\int_0^{a_m} \|e^{a\Delta_D}\|_{\ml(L_q,\Wqq)}\, \|u(0)\|_{L_q}\,\rd a\\
&\qquad+\eta\,\|b_1\|_\infty\int_0^{a_m}\int_0^a  \|e^{(a-\sigma)\Delta_D}\|_{\ml(L_q,\Wqq)}\, \|\alpha_1 u(\sigma)^2+\alpha_2 u(\sigma) v(\sigma)\|_{L_q}\,\rd \sigma\,\rd a\\
&\le\, c_1(M)
\end{split}
\eqnn
and consequently $\|u\|_{\Wq}\le c(M)$. Since $\xi\le M$, we similarly deduce $\|v\|_{\Wq}\le c(M)$.
\end{proof}

Finally, we show that $\mathfrak{B}_3$ connects $\mathfrak{B}_2$ with $\mathfrak{B}_1$ for certain values of $\eta$.
To state the precise result observe that $r(H_{[\alpha_2v_\xi]})$ is a strictly decreasing function of $\xi$ on $(1,\infty)$ according to Lemma~\ref{L1} and Corollary~\ref{C1}.
Since $v_\xi$ depends continuously on $\xi$ in the topology of $\Wq$, we obtain from \cite[II.Lem.5.1.4]{LQPP} that the evolution operator $\Pi_{[\alpha_2v_\xi]}(a,0)$ and hence $\hat{H}_{[\alpha_2v_\xi]}$ depend continuously on $\xi$ in the corresponding operator topologies. Together with the fact that the spectral radius considered as a function $\mk(\Wqq)\rightarrow \R^+$ is continuous (see \cite[Thm.2.1]{Degla09}), we conclude that
\bqn\label{46a}
\big(\xi\mapsto r(H_{[\alpha_2v_\xi]})\big)\in C\big((1,\infty),(0,\infty)\big)\ \,\text{is strictly decreasing}\ .
\eqn
By Theorem~\ref{T1}, the branch $\{(\xi,v_\xi)\,;\, \xi>1\}$ emanates from $(1,0)$ and $r(H_{[\alpha_2v_\xi]})<r(H_{[0]})=1$ thanks to Lemma~\ref{L1} and \eqref{12}, hence $\lim_{\xi\rightarrow 1}r(H_{[\alpha_2v_\xi]})=1$. Defining $N\in (1,\infty]$ by
\bqn\label{ooo}
N:=\frac{1}{\lim\limits_{\xi\rightarrow \infty}r(H_{[\alpha_2v_\xi]})}\ ,%\frac{1}{sup\limits_{\eta\rightarrow \infty}r(\hat{H}_{[-\beta_2 u_\eta]})}\ ,
\eqn
we thus find for any $\eta\in (1,N)$ fixed a unique $\xi_1:=\xi_1(\eta)>1$ with
\bqn\label{ooooo}
\eta=\frac{1}{r(H_{[\alpha_2v_{\xi_1}]})}\ .
\eqn
For values of $\eta$ less than $N$ we can improve Lemma~\ref{L9}:

\begin{lem}\label{L34}
Suppose $b_2$ satisfies \eqref{B}. If $\eta\in (1,N)$, then $\mathfrak{B}_3$ joins up with $\mathfrak{B}_1$ at the point $(\xi_1,0,v_{\xi_1})$.
\end{lem}

\begin{proof}
If $(\xi,u,v)\in \R^+\times\Wqd\times\Wqd$ solves \eqref{7}-\eqref{10}, then \mbox{$v\ge v_\xi$} by Lemma~\ref{L4} while Lemma~\ref{L5}~(ii) shows $1\le\eta r(H_{[\alpha_2v_\xi]})$. Thus, by definition of $N$, if $\eta<N$, then necessarily there must be some $M(\eta)>0$ such that $\xi\le M(\eta)$ for all $(\xi,u,v)\in \mathfrak{B}_3\subset  \R^+\times\Wqd\times\Wqd$. Now Lemma~\ref{L10} together with Lemma~\ref{L9} imply that $\mathfrak{B}_3$ must join up with $\mathfrak{B}_1$, say, at the point $(\hat{\xi},0,v_{\hat{\xi}})$. 

To determine $\hat{\xi}$ note first that, due to Lemma~\ref{L4}, $(\xi,u,v)=(\xi,u,v_\xi+w)\in \R^+\times\Wq^+\times \Wq^+$ solves \eqref{7}-\eqref{10} if and only if $(\xi,u,w)\in \R^+\times\Wq^+\times \Wq^+$ solves
\begin{align}
&\partial_au-\Delta_Du=-\alpha_1u^2-\alpha_2(v_\xi+w) u\ ,& u(0)=\eta U\ ,\label{18aaa}\\
&\partial_aw-\Delta_Dw=-\beta_1w^2-2\beta_1v_\xi w+\beta_2 (v_\xi+w) u\ ,& w(0)=\xi W\ ,\label{18bbb}
\end{align}
where we put
$$
U:=\int_0^{a_m} b_1(a)\, u(a)\,\rd a\ ,\qquad W:=\int_0^{a_m} b_2(a)\, w(a)\,\rd a\ .
$$
Introducing
$$
T:=\big(\partial_a-\Delta_D,\gamma_0)^{-1}\in \ml(\Lq\times\Wqq, \Wq)
$$
and the operators 
$$
\tilde{K}(\xi)(u,w):=\left(\begin{matrix} T(-\alpha_2v_\xi u ,\eta U)\\ T(-2\beta_1v_\xi w+\beta_2 v_\xi u,\xi W)\end{matrix}\right)\ ,\qquad \tilde{R}(u,w):=-\left(\begin{matrix} T(-\alpha_1 u^2-\alpha_2uw,0)\\ T(-\beta_1w^2+\beta_2uw,0)\end{matrix}\right)\ 
$$ 
acting on $(u,w)\in \Wq\times\Wq$, equations \eqref{18aaa}, \eqref{18bbb} are equivalent to
\bqn\label{1999}
(u,w)-\tilde{K}(\xi)(u,w)+\tilde{R}(u,w)=0\ .
\eqn
The operators $\tilde{K}(\xi)$ and $\tilde{R}$ possess the properties stated in Lemma~\ref{L7} (i), (ii). Now, if $((\xi_j,u_j,v_j))_j$ is a sequence in $\mathfrak{B}_3$ converging to $(\hat{\xi},0,v_{\hat{\xi}})$, set $w_j:=v_j-v_{\xi_j}$. As $v_\xi$ depends continuously on $\xi$, formulation \eqref{1999} and the properties of $\tilde{K}(\xi)$ and $\hat{R}$ readily  imply (see, e.g., the proof of \cite[Lem.6.5.3]{LopezGomezChapman} or \cite[Thm.3.1]{BlatBrown2}) that
$$
\frac{(u_j,w_j)}{\|(u_j,w_j)\|_{\Wq\times\Wq}}
$$
converges to some eigenvector $(\phi,\psi)\in\Wq^+\times\Wq^+\setminus\{(0,0)\}$ of $\tilde{K}(\hat{\xi})$ associated to the eigenvalue $1$ and thus satisfying \eqref{18aaa}, \eqref{18bbb} with $\xi=\hat{\xi}$ when higher order terms are neglected:
\begin{align*}
\partial_a\phi-\Delta_D\phi &=-\alpha_2 v_{\hat{\xi}}\phi\ ,& \phi(0)=\eta\,\Phi\ ,\\
\partial_a\psi-\Delta_D\psi &=-2\beta_1v_{\hat{\xi}}\psi+\beta_2 v_{\hat{\xi}}\phi\ ,& \psi(0)=\hat{\xi}\Psi\ .
\end{align*}
Observing that $$1=r(\hat{\xi}\hat{H}_{[\beta_1 v_{\hat{\xi}}]})>r(\hat{\xi}\hat{H}_{[2\beta_1 v_{\hat{\xi}}]})$$  by the analogue of \eqref{sp} and Lemma~\ref{L1}, it follows by a contradiction argument exactly as in the proof of Lemma~\ref{L6} that $\phi\not\equiv 0$.
In particular, this shows that $(1-\eta H_{[\alpha_2 v_{\hat{\xi}}]}) \phi(0)=0$ with $\phi(0)>0$. Hence $\eta^{-1}=r(H_{[\alpha_2 v_{\hat{\xi}}]})$ due to Lemma~\ref{L1} and so $\hat{\xi}=\xi_1$ by~\eqref{ooooo}. This proves the lemma.
\end{proof}

Gathering Lemma~\ref{L9}, Lemma~\ref{L10}, and Lemma~\ref{L34}, the proof of Theorem~\ref{T2} is complete since there is no solution $(\xi,u,v)$ in $\R\times\Wqd\times\Wq^+$ if $\eta\le 1$ and bifurcation of $\mathfrak{B}_3$ at $(\xi_0(\eta),u_\eta,0)$ must be to the right according to Lemma~\ref{L5}. 

\begin{rem}\label{R2}
Note that $\|v_\xi\|_{\Wq}\rightarrow\infty$ as $\xi\rightarrow\infty$ by Theorem~\ref{T1} (in fact: \mbox{$\|v_\xi(0)\|_{\infty}\rightarrow\infty$} by Lemma~\ref{L3}) suggesting that $r(H_{[\alpha_2v_\xi]})$ tends to zero as $\xi$ approaches infinity or, equivalently, that $N= \infty$ in \eqref{ooo} and whence also in Theorem~\ref{T2}.
\end{rem}

\section{Bifurcation for the Parameter $\eta$:  Proof of Theorems \ref{T3} and \ref{T4}}\label{sec4}

This section is dedicated to the proofs of Theorem~\ref{T3} and Theorem~\ref{T4}. We thus regard $\eta$ as bifurcation parameter in \eqref{7}-\eqref{10} and keep $\xi$ fixed. We write $(\eta,u,v)$ for a solution to \eqref{7}-\eqref{10} and suppress $\xi$ since no confusion seems likely.\\

\subsection{Proof of Theorem \ref{T3}}\label{5.3}

The argument used in the proof of Theorem~\ref{T2} is similar to that for the proof of Theorem~\ref{T3} and we thus merely sketch the latter pointing out the main modifications to be made. Let $\xi>1$ be fixed. Then Theorem~\ref{T1} ensures the existence of the semi-trivial branches
$$
\mathfrak{S}_1=\{(\eta,u_\eta,0)\,;\, \eta>1\}\ ,\quad \mathfrak{S}_2=\{(\eta,0,v_\xi)\,;\, \eta\in\R\}
$$
of solutions to \eqref{7}-\eqref{10} in $\R\times\Wq^+\times \Wq^+$. Recall from \eqref{18aaa}, \eqref{18bbb} that $$(\eta,u,v)=(\eta,u,v_\xi+w)\in \R^+\times\Wq^+\times \Wq^+$$ solves \eqref{7}-\eqref{10} provided that $(\eta,u,w)\in \R^+\times\Wq^+\times \Wq^+$ satisfies
\bqn\label{19999}
(u,w)-\hat{K}(\eta)(u,w)+\hat{R}(u,w)=0\ ,
\eqn
with
$$
\hat{K}(\eta)(u,w):=\left(\begin{matrix} \hat{Z}_1(0,\eta U)\\ \hat{Z}_2(\beta_2v_\xi u,\xi V)\end{matrix}\right)\ ,\qquad \hat{R}(w,v):=-\left(\begin{matrix} \hat{Z}_1(-\alpha_1 u^2-\alpha_2wu,0)\\ \hat{Z}_2(-\beta_1v^2+\beta_2uw,0)\end{matrix}\right)\ 
$$
for $(w,v)\in \Wq\times\Wq$, where $\hat{Z}_j\in \ml(\Lq\times\Wqq, \Wq)$ are given by
$$
\hat{Z}_1:=\big(\partial_a-\Delta_D+\alpha_2 v_\xi ,\gamma_0)^{-1}\ ,\quad \hat{Z}_2:=\big(\partial_a-\Delta_D+2\beta_1 v_\xi ,\gamma_0)^{-1}\ .
$$
The operators $\hat{K}(\eta)$ and $\hat{R}$ possess the properties stated in Lemma~\ref{L7} (i), (ii). Analogously to Lemma~\ref{L6} one shows that, given $\eta\in\R$, if $\mu\ge 1$ is an eigenvalue of $\hat{K}(\eta)$ with eigenvector $(u,v)\in\Wq\times\Wq$, then $\eta\not=0$, and $\mu/\eta$ is an eigenvalue of~$H_{[\alpha_2 v_\xi]}$ with eigenvector $u(0)\in\Wqq$. As in Lemma~\ref{L8}, if
\bqn\label{oo}
\eta_0(\xi):=\frac{1}{r(H_{[\alpha_2 v_\xi]})}>1\ ,
\eqn
then $\mathrm{Ind}(0,\hat{K}(\eta))$ changes sign as $\eta$ crosses $\eta_0(\xi)$, and $\mu_*=1$ is a simple eigenvalue of $\hat{K}(\eta_0(\xi))$.
%More precisely, we have
%$$
%\Wq\times\Wq=\mathrm{ker}\big(1-\hat{K}(\eta_0(\xi))\big)\oplus \mathrm{rg}\big(1-\hat{K}(\eta_0(\xi))\big)
%\ ,\qquad \mathrm{ker}\big(1-\hat{K}(\eta_0(\xi))\big)=\mathrm{span}\{(\hat{\phi}_*,\hat{\psi}_*)\}
%$$
%with 
%$$
%\hat{\phi}_*=\hat{Z}_1(0,\hat{\phi}_0)\in\Wqd\ ,\quad \hat{\phi}_0\in\mathrm{int}(\Wqqp)\ ,\quad \hat{\phi}_*\in\Wqd\ .
%$$
Invoking again \cite[Cor.6.3.2]{LopezGomezChapman} we obtain a connected branch $\mathfrak{S}_3'\subset\R\times\Wq\times\Wq$ of solutions to \eqref{7}-\eqref{10} bifurcating from $(\eta_0(\xi),0,v_\xi)$. By definition, $(\eta_0(\xi),0,v_\xi)\not\in\mathfrak{S}_3'$. Further, $\mathfrak{S}_3'$ satisfies the alternatives
\begin{itemize}
\item[(i)] $\mathfrak{S}_3'$ is unbounded in $\R\times\Wq\times\Wq$, or
\item[(ii)] $\mathfrak{S}_3'$ contains a point $(\eta,0,v_\xi)$ such that $1$ is an eigenvalue of $\hat{K}(\eta)$ but $\eta\not=\eta_0(\xi)$, or
\item[(iii)] $\mathfrak{S}_3'$ contains a point $(\eta,u,v_\xi+w)$ with $(u,w)\in\mathrm{rg}(1-\hat{K}(\eta_0(\xi)))\setminus\{(0,0)\}$.
\end{itemize}
Moreover, points on $\mathfrak{S}_3'$ close to $(\eta_0(\xi),0,v_\xi)$ belong to $\R^+\times\Wqd\times \Wqd$. In fact, we have:

\begin{lem}\label{L14}
Let $\mathfrak{S}_3:=\mathfrak{S}_3'\cap \big(\R^+\times\Wqd\times \Wqd\big)$. Then
$\mathfrak{S}_3=\mathfrak{S}_3'$.
\end{lem}

\begin{proof}
Suppose $\mathfrak{S}_3$ is a proper subset of $\mathfrak{S}_3'$. Then there are 
$$
(\eta_j,u_j,v_j)\in \R^+\times \Wqd\times\Wqd\ ,\qquad (\eta,u,v)\in \mathfrak{S}_3'\ ,\quad (u,v)\notin \Wqd\times\Wqd
$$ 
with 
$$
(\eta_j,u_j,v_j)\rightarrow (\eta,u,v)\quad\text{in}\quad  \R\times \Wq\times\Wq\ .
$$
As \eqref{emb} ensures $u(0)\ge 0$ and $v(0)\ge 0$, whence $u,v\in\Wq^+$ by \eqref{darst1}, \eqref{darst2}, the only possibility that $(u,v)$ does not belong to $\Wqd\times\Wqd$ is that $u\equiv 0$ or $v\equiv 0$. However, since $v_j\in\Wqd$ and thus $v_j(a)\ge v_\xi(a)$ for $a\in J$ owing to Lemma~\ref{L4}, $v\in\Wqd$ and so necessarily $u\equiv 0$.
Hence, by the uniqueness statement in Theorem~\ref{T1}, we deduce $v=v_\xi$. But then, $(\eta,0,0)$ is a bifurcation point for \eqref{19999} and it follows from \cite[Lem.6.5.3]{LopezGomezChapman} exactly as in the proof of Lemma~\ref{L9} that this implies $\eta=\eta_0(\xi)$. Thus
$(\eta,u,v)=(\eta_0(\xi),0,v_\xi)$ what is not possible.
\end{proof}

Now, as $\mathfrak{S}_3'=\mathfrak{S}_3\subset \R^+\times\Wqd\times \Wqd$, alternative (ii) above is impossible, while alternative (iii) can be ruled out by using an argument analogous to that in the proof of Lemma~\ref{L9}. Therefore, $\mathfrak{S}_3$ is unbounded in $\R^+\times\Wqd\times \Wqd$. That bifurcation at $(\eta_0(\xi),0,v_\xi)$ is to the right, is a consequence of Lemma~\ref{L5}~(ii). Finally, let $b_2$ satisfy \eqref{B} and suppose there is some $M>0$ with $\eta\le M$ for $(\eta,u,v)\in\mathfrak{S}_3$. Combining Lemma~\ref{L3} and Lemma~\ref{L4} we obtain $\|u(a)\|_\infty\le \kappa M^2$ for $a\in J$ and we may then proceed as in Lemma~\ref{L10} to show that $\|u\|_{\Wq}+\|v\|_{\Wq}\le c(M)$ for some $c(M)>0$ independent of $(\eta,u,v)\in\mathfrak{S}_3$. This completes the proof of Theorem~\ref{T3}.

\subsection{Proof of Theorem~\ref{T4}}\label{5.33}

We now focus on the proof of Theorem~\ref{T4}. Let $\xi<1$. Then Theorem~\ref{T1} implies that
$$
\mathfrak{S}_1=\{(\eta,u_\eta,0)\,;\, \eta>1\}\
$$
is the only semi-trivial branch of solutions to \eqref{7}-\eqref{10}. The same arguments leading to \eqref{46a} show that
\bqn\label{46}
\big(\eta\mapsto r(\hat{H}_{[-\beta_2 u_\eta]})\big)\in C\big((1,\infty),(0,\infty)\big)\ \,\text{is strictly increasing}\ 
\eqn
with $\lim_{\eta\rightarrow 1}r(\hat{H}_{[-\beta_2 u_\eta]})=1$. Defining $\delta\in [0,1)$ by
\bqn\label{oooa}
\delta:=\frac{1}{\lim\limits_{\eta\rightarrow \infty}r(\hat{H}_{[-\beta_2 u_\eta]})}\ ,%\frac{1}{sup\limits_{\eta\rightarrow \infty}r(\hat{H}_{[-\beta_2 u_\eta]})}\ ,
\eqn
it follows that for any $\xi\in (\delta,1)$ fixed we find a unique $\eta_1:=\eta_1(\xi)>1$ with
\bqn\label{ooooa}
\xi=\frac{1}{r(\hat{H}_{[-\beta_2 u_{\eta_1}]})}\ .
\eqn
To demonstrate that local bifurcation from $\mathfrak{S}_1$ occurs at the point $(\eta_1,u_{\eta_1},0)$, we apply the theorem of Crandall-Rabinowitz \cite{CrandallRabinowitz}.
Introducing
$$
T=\big(\partial_a-\Delta_D,\gamma_0)^{-1} \in \ml(\Lq\times\Wqq, \Wq)\ ,
$$
we observe from \eqref{18a}, \eqref{18b} that $(\eta,u,v)=(\xi,u_\eta-w,v)\in \R^+\times\Wq^+\times \Wq^+$ solves \eqref{7}-\eqref{10} if and only if $(\eta,w,v)\in \R^+\times\Wq^+\times \Wq^+$ with $w\le u_\eta$ is a zero of the function
$$
G(\eta,w,v):=\left(\begin{matrix} w-T(\alpha_1w^2-2\alpha_1 u_\eta w+\alpha_2 (u_\eta-w) v\, ,\, \eta W)\\ v-T(-\beta_1 v^2 +\beta_2 (u_\eta-w) v\, ,\, \xi V)\end{matrix}\right)\ ,
$$
where we again agree here and for the remainder of this subsection upon the slight abuse of notation
$$
W:=\int_0^{a_m} b_1(a)\, w(a)\,\rd a\ ,\qquad V:=\int_0^{a_m} b_2(a)\, w(a)\,\rd a\ ,
$$
being used for other capital letters as well since it will always be clear from the context, which of the profiles $b_1$ and $b_2$ we mean.
Theorem~\ref{T1} warrants $$G\in C^1((1,\infty)\times\Wq\times\Wq,\Wq\times\Wq)$$ with partial Frech\'et derivatives at $(\eta,w,v)=(\eta_1,0,0)$ given by
$$
G_{(w,v)}(\eta_1,0,0)(\phi,\psi)=\left(\begin{matrix} \phi-T(-2\alpha_1 u_{\eta_1} \phi+\alpha_2 u_{\eta_1} \psi\, ,\, {\eta_1} \Phi)\\ \psi-T(\beta_2 u_{\eta_1} \psi\, ,\, \xi \Psi)\end{matrix}\right)
$$
and
$$
G_{\eta,(w,v)}(\eta_1,0,0)(\phi,\psi)=\left(\begin{matrix} -T(-2\alpha_1 u_{\eta_1}' \phi+\alpha_2 u_{\eta_1}' \psi\, ,\,  \Phi)\\ -T(\beta_2 u_{\eta_1}' \psi\, ,\, 0)\end{matrix}\right)\ ,\qquad u_\eta':=\frac{\partial}{\partial\eta} u_\eta\ 
$$
for $(\phi,\psi)\in \Wq\times\Wq$. We claim that the kernel of $G_{(w,v)}(\eta_1,0,0)$ is one-dimensional. Indeed, for $(\phi,\psi)\in \mathrm{ker}(G_{(w,v)}\big(\eta_1,0,0))$ we have
\begin{align}
\partial_a\phi-\Delta_D\phi &=-2\alpha_1 u_{\eta_1} \phi+\alpha_2 u_{\eta_1} \psi\ ,& \phi(0)=\eta_1\,\Phi\ ,\label{t}\\
\partial_a\psi-\Delta_D\psi &=\beta_2 u_{\eta_1} \psi\ ,& \psi(0)=\xi\Psi\ ,\label{tt}
\end{align}
and so an argument similar to that used in the proof of Lemma~\ref{L8} (with $\mu_*=1$) shows that $(\phi,\psi)$ must be a scalar multiple of $(\phi_*,\psi_*)\in \Wqd\times\Wqd$, where
$$
\psi_*:=\Pi_{[-\beta_2 u_{\eta_1}]}(\cdot,0)\,\psi_0\ ,\qquad \psi_0\in\mathrm{ker}\big(1-\xi\hat{H}_{[-\beta_2 u_{\eta_1}]}\big)\cap \mathrm{int}(\Wqqp)\ ,
$$
and
$$
\phi_*:=\Pi_{[2\alpha_1 u_{\eta_1}]}(\cdot,0)\phi_0+ N\psi_*\ ,\qquad \phi_0:={\eta_1}\big(1-{\eta_1}H_{[2\alpha_1 u_{\eta_1}]}\big)^{-1} \int_0^{a_m} b_1(a)(N\psi_*)(a)\,\rd a\  ,
$$
with
$$
(N\psi_*)(a):=
{\alpha_2}\int_0^a \Pi_{[2\alpha_1 u_{\eta_1}]}(a,\sigma)\,\big(u_{\eta_1}(\sigma) \psi_*(\sigma)\big)\,\rd \sigma\ , \quad a\in J\ .
$$
Thus,
$$
\mathrm{ker}(G_{(w,v)}\big(\eta_1,0,0)\big)=\mathrm{span}\big\{(\phi_*,\psi_*)\big\}\ .
$$ 
As the derivative of $G$ has the form $G_{(w,v)}(\eta_1,0,0)=1-\hat{T}$ with a compact operator $\hat{T}$ (see \eqref{comp}), we also get from this that the codimension of $\mathrm{rg}\big(G_{(w,v)}(\eta_1,0,0)\big)$ equals one.
To check the transversality condition of \cite{CrandallRabinowitz}, suppose that
$$
G_{\eta,(w,v)}(\eta_1,0,0)(\phi_*,\psi_*)\in\mathrm{rg}\big(F_{(w,v)}(\eta_1,0,0)\big)
$$
and let $v\in\Wq$ be with 
$$
v-T(\beta_2u_{\eta_1} v, \xi V)=-T(\beta_2 u_{\eta_1}'\psi_*, 0)\ .
$$
Since $\psi_0\in\mathrm{int}(\Wqqp)$ we may choose $\tau>0$ such that $\tau\psi_0-v(0)\in\mathrm{int}(\Wqqp)$. Setting $p:=\tau \psi_*-v$ and observing $\psi= T(\beta_2 u_{\eta_1}\psi,\xi\Psi)$,
it follows $$p=T(\beta_2u_{\eta_1} p+\beta_2 u_{\eta_1}'\psi_*,\xi P)\ ,$$ that is,
$$
\partial_a p-\Delta_D p=\beta_2u_{\eta_1} p+\beta_2 u_{\eta_1}'\psi_*\ ,\quad p(0)=\xi P \ ,
$$
from which
$$
\big(1-\xi \hat{H}_ {[-\beta_2u_{\eta_1}]}\big)\, p(0)\,=\,\xi \beta_2\int_0^{a_m}b_2(a)\,\int_0^a\Pi_{[-\beta_2u_{\eta_1}]}(a,\sigma)\,\big(u_{\eta_1}'(\sigma)\,\psi_*(\sigma)\big)\,\rd \sigma\,\rd a\ .
$$
This contradicts that fact that this equation has no positive solution $p(0)=\tau\psi_0-v(0)$ owing to \cite[Cor.12.4]{DanersKochMedina} and \eqref{ooooa} since the right hand side belongs to $\mathrm{int}(\Wqqp)$ thanks to \eqref{11}, Corollary~\ref{C2}, and the strong positivity of the operator $\Pi_{[-\beta_2u_{\eta_1}]}(a,\sigma)$ for $0\le\sigma< a\le a_m$. Consequently, 
$$
G_{\eta,(w,v)}(\eta_1,0,0)(\phi_*,\psi_*)\notin\mathrm{rg}\big(G_{(w,v)}(\eta_1,0,0)\big)\ .
$$
We are thus in a position to apply \cite[Thm.1.7]{CrandallRabinowitz} and deduce the existence of a branch $\mathfrak{S}_4'$ of solutions to \eqref{7}-\eqref{10} bifurcating from $(\eta_1,u_{\eta_1},0)$, where $\mathfrak{S}_4'$ is of the form
$$
\mathfrak{S}_4'=\big\{\big(\eta(\ve),\ve(\phi_*,\psi_*)+\ve(\theta_1(\ve),\theta_2(\ve))\big)\, ;\, \vert\ve\vert<\ve_0\big\}
$$
for some $\ve_0>0$ with $\eta(0)=\eta_1$, $\theta_j(0)=0$, and $\theta_j\in C((-\ve_0,\ve_0),\Wq)$. Clearly, it follows from \eqref{darst1} and \eqref{darst2} that, if $\ve_0>0$ is sufficiently small, then the points $(\eta,u,v)$ of $\mathfrak{S}_4'$ associated to values $\ve\in (0,\ve_0)$ in the representation above satisfy $(u,v)\in\Wqd\times\Wqd$ since both $\phi_*(0)=\phi_0$ and $\psi_*(0)=\psi_0$ belong to $\mathrm{int}(\Wqqp)$. Letting 
$$
\mathfrak{S}_4:=\mathfrak{S}_4'\cap (\R^+\times\Wqd\times\Wqd)
$$ 
it is easy to check that $\mathfrak{S}_4$ bifurcates from $(\eta_1,u_{\eta_1},0)$ to the right in view of \eqref{17}, \eqref{46}, and \eqref{ooooa}. The proof of Theorem~\ref{T4} is therefore complete.\\

\begin{rem}\label{R9}
Similarly as in Remark~\ref{R2} we conjecture that $\lim\limits_{\eta\rightarrow \infty}r(\hat{H}_{[-\beta_2 u_\eta]})=\infty$, whence $\delta=0$ in Theorem~\ref{T4}.
\end{rem}

%\section*{Acknowledgement}

\end{document}